\documentclass[a4paper]{amsart}

\usepackage{amssymb}
\usepackage{amsthm}  
\usepackage{amsmath} 
\usepackage{amscd} 
\usepackage{float}
\usepackage[all]{xypic}
\usepackage{url}
\usepackage{graphicx}
\usepackage{hyperref}
\usepackage{color}
\hypersetup{linktocpage}

\newcommand{\om}{\omega}

\newcommand{\si}{\sigma}

\newcommand{\ba}{\mathcal{G}}

\newcommand{\fg}{\mathfrak g}

\newcommand{\fp}{\mathfrak p}

\newcommand{\fk}{\mathfrak k}

\newcommand{\fh}{\mathfrak h}
\newcommand{\fn}{\mathfrak n}

\newcommand{\fm}{\mathfrak m}

\newcommand{\id}{{\rm id}}

\theoremstyle{plain}

\theoremstyle{plain}
\newtheorem{thm*}{Theorem}[section]
\newtheorem{prop*}[thm*]{Proposition}
\newtheorem{lem*}[thm*]{Lemma}
\newtheorem{cor*}[thm*]{Corollary}

\theoremstyle{definition}
\newtheorem{def*}{Definition}

\theoremstyle{remark}
\newtheorem{rem*}{Remark}

\begin{document}
\title[Equivalence problem for 2--nondegenerate CR geometries]{On equivalence problem for 2--nondegenerate CR geometries with simple models}
\author{Jan Gregorovi\v c}
\address{Faculty of Mathematics, University of Vienna, Oskar Morgenstern Platz 1, 1090 Wien, Austria}
 \email{jan.gregorovic@seznam.cz}
 \thanks{The author was supported by the project P29468 of the Austrian Science Fund (FWF)}
\maketitle
\begin{abstract}
In this article, we solve the equivalence problem for 2--nondegenerate CR geometries that have (at every point) a homogeneous space $G/H$ as a maximally symmetric model for $G$ simple real Lie group of CR automorphisms. This completes the classification of real submanifolds in complex space that are maximally symmetric models with a real simple CR automorphism group. In particular, we construct (local) embeddings of these models into complex space.
\end{abstract}

\tableofcontents
\section{Introduction}

A way toward the solution of the problem of the biholomorphic equivalence for real submanifolds in the complex space is to compare the induced CR geometries with appropriate model CR geometries. In the case of the Levi--nondegenerate real hypersurfaces in $\mathbb{C}^N$, the maximally symmetric models are quadrics that can be identified (depending on the signature) with (the CR geometries on) the homogeneous spaces $G/H=SU(p+1,N-p)/P$, where $SU(p+1,N-p)$ is the (real simple) CR automorphism group of the quadric and the stabilizer $P$ is a particular parabolic subgroup of $SU(p+1,N-p)$. This allowed Cartan \cite{cartan} for $N=2$ and Tanaka \cite{Ta62,Ta70,Ta76} and Chern and Moser \cite{chern} for $N\geq 2$ to solve the equivalence problem for Levi--nondegenerate CR hypersurfaces. Of course, there are also Levi--nondegenerate CR hypersurfaces that have a non--maximal real simple CR automorphism group, see \cite{Gr12}.

In the case of the Levi--nondegenerate CR geometries of codimension higher than $1$, the equivalence problem is solved under the assumption that we can compare the CR geometry with the same (maximally symmetric) model at every point. Under this assumption, the Tanaka's prolongation procedure \cite{Ta70} solves the equivalence problem for such real submanifolds in the complex space. Nevertheless, this solution is sufficient to classify the maximally symmetric models of Levi--nondegenerate CR geometries with semisimple CR automorphism group. We review this classification in detail in the following Section \ref{uv1}. Let us emphasize at this point that all of the models in the classification fit in the class of parabolic geometries c.f. \cite{parabook}, which will also play an important role in the article. Particular classes of such CR geometries were studied  in detail by Schmalz and Slovak \cite{SS00,SS12}.

In the case of general Levi--degenerate CR geometries, the equivalence problem is still open. Recently, several solutions different of the equivalence problem were found for the everywhere 2--nondegenerate CR hypersurfaces in $\mathbb{C}^3$ by Ebenfelt \cite{Eb06}, Isaev and Zaitsev \cite{IZ13}, Pocchiola \cite{Poc13}, Medori and Spiro \cite{MS14,MS15} and Kossovskyi and Kolar \cite{KK19}. The 2--nondegeneracy is the strongest nondegeneracy condition for Levi--degenerate CR geometries c.f. \cite{ber0,Fr74,Fr77} and we recall the precise definition later. Further, Porter \cite{Por15} also found a solution of the equivalence problem for many classes of everywhere 2--nondegenerate CR hypersurfaces in $\mathbb{C}^4$. Finally, Porter and Zelenko \cite{PZ17} solve the equivalence problem for many classes of everywhere 2--nondegenerate CR hypersurfaces with one dimensional Levi kernel. As we show later, these results provide only few examples of maximally symmetric models of 2--nondegenerate CR submanifolds that have real simple CR automorphism group, but not a complete classification.

In this article, we solve the equivalence problem for 2–nondegenerate CR geometries  in {\em arbitrary} dimension satisfying the following natural requirement: the CR geometry is modeled, at each point, by a homogeneous space $G/H$ with the model being maximally symmetric and having a real simple CR automorphism group. We do so by extending to the 2-nondegenerate setting the {\em Tanaka’s prolongation procedure} \cite{Ta70}. An overview of our procedure and detailed results are given in Section \ref{uv2}.

Let us remark that it is also possible to answer the question of what are the maximally symmetric model with semisimple CR automorphism group algebraically. This way would generalize the work of Santi \cite{Sa15} on models of 2--nondegenerate CR hypersurface, but we know from results mentioned above that there are obstructions for this approach to provide the solution of the equivalence problem. We also obtain this algebraic result, see Remark \ref{remsem}, on the way to solve the equivalence problem. 

Finally, we show (Theorem \ref{neces3}) that there are no maximally symmetric models of CR geometries that are nondegenerate of order $\ell\geq 3$  that would have semisimple CR automorphism group and thus we  completed the (algebraic) classification of maximally symmetric models with semisimple CR automorphism group.

\subsection{Semisimple maximally symmetric models of  Levi--nondegenerate CR geometries}\label{uv1}

A (bracket generating) filtration on a smooth manifold $M$ is a family of successive subbundles $$0=T^0M\subset T^{-1}M\subset \dots \subset T^iM\subset T^{i-1}M\subset \dots \subset T^{\mu}M=TM$$ such that $T^{-1}M$ generates $TM$ (via the Lie bracket). The pair $(M,T^iM)$ is called a filtered manifold if a Lie bracket of sections of $T^iM$ and $T^jM$ is a section of $T^{i+j}M$. The bracket of vector fields on filtered manifold defines an algebraic Lie bracket $\mathcal{L}_x$ on the associated graded tangent bundle 
 $$gr(T_xM):=\bigoplus_i gr_i(T_xM),\ gr_i(T_xM):= T_x^{-i}M/T_x^{-i+1}M$$
at each $x\in M$. In particular, $(gr(T_xM),\mathcal{L}_x)$ is a graded Lie algebra.

The filtered manifold $(M,T^iM)$ is called regular with symbol $\fm$ if the graded Lie algebra $(gr(T_xM),\mathcal{L}_x)$ is for each $x\in M$ isomorphic to a graded Lie algebra $\fm=\oplus_{i<0} \fm_i$. If $\fm_{-1}$ does not contain an non--trivial ideal $\fk$ of $\fm$, then the symbol $\fm$ is called non--degenerate.

A geometric structure on a regular filtered manifold with symbol $\fm$ is a reduction of the frame bundle of graded isomorphisms of $(gr(T_xM),\mathcal{L}_x)$ with $\fm$ to a subgroup $G_0$ of the group $Aut_0(\fm)$ of grading preserving automorphisms of $\fm$.

\begin{def*}
A CR geometry with Levi--Tanaka algebra $(\fm,I)$ on a regular filtered manifold $(M,T^iM)$ with symbol $\fm$ is
\begin{enumerate}
\item a complex structure $I$ on $\fm_{-1}$ such that $[I(X),I(Y)]=[X,Y]$ holds for all $X,Y\in \fm_{-1}$, and
\item a reduction of the group $Aut_0(\fm)$ of grading preserving automorphisms of $\fm$ to a subgroup $G_{0,I}$ of $Aut_0(\fm)$ consisting of the elements preserving $I$.
 \end{enumerate}

The distribution $\mathcal{K}\subset T^{-1}M$ corresponding to the maximal non--trivial ideal $\fk$ of $\fm$ in $\fm_{-1}$ is called a Levi kernel. A CR geometry is called Levi--nondegenerate if the symbol $\fm$ is non--degenerate, i.e., $\mathcal{K}=0$.
\end{def*}

The Levi--nondegenerate hypersurfaces $M\subset \mathbb{C}^N$ provide typical examples of CR geometries. The complex tangent space $T^{-1}M=TM\cap i(TM)\subset T\mathbb{C}^N$ together with the Levi bracket form a regular filtered manifold with a symbol that is the Heisenberg algebra. The complex structure on $T^{-1}M$ provides a reduction of the group $CSp(2N-2,\mathbb{R})$ of the grading preserving automorphisms of the Heisenberg algebra to $CSU(p,N-p-1)$, where $(p,N-p-1)$ is the signature of the Levi form.

In general, the nilpotent Lie group $\exp(\fm)$ together with the left invariant filtration and the geometric structure induced by the grading of $\fm$ and reduction to $G_0$ is a (local) model of a $G_0$--structure on a regular filtered manifold with nondegenerate symbol $\fm$ in the Tanaka's prolongation theory \cite{Ta70}. Moreover, Tanaka proved that the Lie algebra of (local) infinitesimal automorphisms, i.e., vector fields preserving the filtration and the geometric structure, is finite dimensional and can be bound by dimension of the Lie algebra of infinitesimal automorphisms of the (local) model. Starting with the nondegenerate symbol $\fm$ and the Lie algebra $\fg_0$ of $G_0$, the Lie algebra of infinitesimal automorphisms of the (local) model is a graded Lie algebra $\fg=\fg_{\mu}\oplus \dots \oplus \fg_{-1}\oplus \fg_{0}\oplus \fg_{1}\oplus \dots$, where $\fg_i=\fm_i$ for $i<0$ and $\fg_i$ for $i>0$ can be computed as
$$\fg_i=\{f\in \oplus_{j<0}\fg_j^*\otimes \fg_{j+i}:\ f([X,Y])=[f(X),Y]+[X,f(Y)] \rm{\ for\ all\ } X,Y\in \fm\}.$$
The graded Lie algebra $\fg$ is usually called the Tanaka prolongation of $\fm\oplus \fg_0$.

If the Tanaka prolongation $\fg$ of $\fm\oplus \fg_0$ is a semisimple Lie algebra, then the non--negative part $\fp=\bigoplus_{i\geq0}\fg_i$ of $\fg$ is a parabolic subalgebra of $\fg$ and there is a global model $G/P$ containing the local model $\exp(\fm)$ as an open subset, i.e., $G/P$ is the maximally symmetric model. The corresponding $G_0$--structures on regular filtered manifolds with the symbol $\fm$ are usually called parabolic geometries, cf. \cite{parabook}. The case of Levi--nondegenerate hypersurfaces is among them because the Lie algebra $\frak{su}(p+1,N-p)$ is the Tanaka prolongation of the symbol Heisenberg algebra plus $\frak{csu}(p,N-p-1)$.

Apart few exceptions that are not related to CR geometries, for each parabolic subgroup $P$ of $G$ there is a grading $\fg_i$ of the Lie algebra $\fg$ of $G$ such that $\fp=\bigoplus_{i\geq0}\fg_i$ is the Lie algebra of $P$ and $\fg$ is the Tanaka prolongation of $\fg_-\oplus \fg_0$, where $\fg_-$ is the negative part of the grading. Thus, it suffices to find parabolic geometries that admit a complex structure $I$ on $\fg_{-1}$ preserved by $G_0$ such that $[I(X),I(Y)]=[X,Y]$ holds for all $X,Y\in \fg_{-1}$. This was done in \cite{MS} or \cite{AMT06}. The following theorem summarizes these classification results in a way that allows simple comparison with our main result Theorem \ref{main}.

\begin{thm*}\label{nondegclas}
Let $G$ be a semisimple Lie group and $H$ a Lie subgroup of $G$. A homogeneous space $G/H$ is a maximally symmetric model of a Levi--nondegenerate CR geometry if and only if $H=P$ is a parabolic subgroup of $G$ and there is a bigrading $\fg_{a,b}$ of the complexification of $\fg$ such that
\begin{enumerate}
\item the representation of $G_0$ on $\fg_{-1}$ is complex
\item $\fg_a\otimes \mathbb{C}=\oplus_b \fg_{a,b}$ holds for the grading $\fg_i$ of $\fg$ corresponding $P$,
\item $\fg_{0}\otimes \mathbb{C}=\fg_{0,0}$,
\item $\fg_{-1}\otimes \mathbb{C}=\fg_{-1,-1}\oplus \fg_{-1,0}$,
\item $\fg_{-2}\otimes \mathbb{C}=\fg_{-2,-1}.$
\end{enumerate}
\end{thm*}

\subsection{Summary of the results}\label{uv2}

In the setting of filtered manifolds $(M,T^iM)$, Freeman \cite{Fr74,Fr77} defines 2--nondegeneracy of CR geometries with Levi--Tanaka algebra $(\fm,I)$ using the following observation:

The complex antilinear part of the Levi--bracket of section of $\mathcal{K}$ and section of $T^{-1}M$ defines a linear map $\iota_x: \mathcal{K}_x\to \frak{gl}(T^{-1}_xM/\mathcal{K}_x)$. The CR geometry is 2--nondegenerate at $x$ if $\iota_x$ is injective.

Since we are only interested in 2--nondegenerate CR geometries on regular filtered manifold $(M,T^iM)$ with symbol $\fm$, we identify the maximal non--trivial ideal $\fk$ of $\fm$ in $\fm_{-1}$ with $\mathcal{K}_x$ and $\frak{gl}(T^{-1}_xM/\mathcal{K}_x)$ with $\frak{gl}(\fm_{-1}/\fk)$. We say that the graded nilpotent Lie algebra $\fg_{-}:=\fm/\fk$ with restriction of $I$ (to the quotient) is the nondegenerate part of the Levi--Tanaka algebra $(\fm,I)$, i.e., $$\fg_{i}=\fm_i {\rm \ for\ } i<-1 {\rm \ and\ }\fg_{-1}=\fm_{-1}/\fk.$$

\begin{def*}\label{def2dg}
A CR geometry with Levi--Tanaka algebra $(\fm,I)$ on a regular filtered manifold $(M,T^iM)$ is a called 2--nondegenerate if $\iota_x: \fk\to \frak{gl}(\fg_{-1})$ is injective for all $x\in M$. We say that the CR geometry is regularly 2--nondegenerate with  (regular) second--order Levi--Tanaka algebra $(\fg_-,I,\fk)$ if
\begin{enumerate}
\item the CR geometry is 2--nondegenerate and the image $\fk\subset \frak{gl}(\fg_{-1})$ of $\iota_x$ does not depend on $x$,
\item the image $\fk$ is contained in the Lie algebra $\frak{der}_0(\fg_-)$ of grading preserving derivations of $\fg_-$, and
\item $[[\fk,\fk],\fk]\subset \fk$ holds for the bracket in $\frak{der}_0(\fg_-)$.
\end{enumerate}
\end{def*}

In Section \ref{sec2}, we discuss the second--order Levi--Tanaka algebras of 2--nondegenerate CR submanifolds, in detail. We remark that as in the case of Levi--nondegenerate CR geometries of codimension higher than $1$, the regularity is (in general) not a generic assumption. Nevertheless, this does not restrict us to obtain the following theorem that summarizes our main results.

\begin{thm*}\label{main}
Let $G$ be a real simple Lie group and $H$ a Lie subgroup of $G$. A homogeneous space $G/H$ is a maximally symmetric model of a regularly 2--nondegenerate CR geometry if and only if $H$ is a subgroup of a parabolic subgroup $P$ of $G$ and there is a bigrading $\fg_{a,b}$ of the complexification of $\fg$ such that
\begin{enumerate}
\item $P/H=(G_0\rtimes \exp(\fp_+)/G_{0,I}\rtimes \exp(\fp_+))=G_0/G_{0,I}$ is a (pseudo)--Hermitian symmetric space,
\item $\fg_a\otimes \mathbb{C}=\sum_b \fg_{a,b}$ holds for the grading $\fg_i$ of $\fg$ corresponding $P$,
\item $\fg_{0,I}\otimes \mathbb{C}=\fg_{0,0},$
\item $\fg_{0}\otimes \mathbb{C}=\fg_{0,-1}\oplus \fg_{0,0}\oplus \fg_{0,1}$,
\item $\fg_{-1}\otimes \mathbb{C}=\fg_{-1,-1}\oplus \fg_{-1,0}$,
\item $\fg_{-2}\otimes \mathbb{C}=\fg_{-2,-1}.$
\end{enumerate}
\end{thm*}

In Section \ref{sec31}, we will show in Proposition \ref{neces} that the second--order Levi--Tanaka algebras of semisimple maximally symmetric models $G/H$ of 2--nondegenerate CR submanifolds are regularly 2--nondegenerate and we obtain as a consequence of the (semi)simplicity a bigrading $\fg_{a,b}$ satisfying the conditions of the Theorem \ref{main}. This proves one implication in the Theorem \ref{main} and we obtain the following classification of such bigradings in Section \ref{sec33}.

\begin{thm*}\label{mainclass}
A quadruple $G,H,P,\fg_{a,b}$ satisfies the conditions of Theorem \ref{main} if and only if there is entry in Tables \ref{realclasA}, \ref{realclasB} such that $\fg$ is the Lie algebra of $G$, $\fg_{0,I}'$ is the component of $\fg_{0,I}$ acting effectively on $\fk$, the simple roots in the ordered sets $\Sigma_1,\Sigma_2$ determine the bigrading $\fg_{a,b}$, where $\Sigma_1$ satisfy the usual restrictions c.f. \cite[Section 3.2.9]{parabook} and $\Sigma_1,\Sigma_2$ satisfy the additional restrictions specified in the tables.
\noindent
\begin{table}[h!]\caption{Classification of classical simple maximally symmetric models of 2--nondegenerate CR geometries}\label{realclasA}
\begin{tabular}{|c|c|c|}
\hline
$\mathfrak{g}$ & $\Sigma_1$ & restrictions\\
$\fg_{0,I}'$& $\Sigma_2$ &  restrictions \\
\hline
$\mathfrak{sl}(n+1,\mathbb{R})$&$\{\alpha_r,\alpha_{r+2s}\}$ & \\
$\frak{gl}(s,\mathbb{C})$ &$\{\alpha_{r+s}\}$& \\
\hline
$\mathfrak{sl}(n+1,\mathbb{H})$&$\{\alpha_{2r},\alpha_{2s}\}$   & \\
$\frak{gl}(s-r,\mathbb{C})$&$\{\alpha_{r+s}\}$& \\
\hline
$\mathfrak{su}(p,n+1-p)$&$\{\alpha_r,\alpha_{n-r}\}$ & $r<s<n-r$ \\
$\frak{su}(q,s-r-q)\oplus \frak{u}(p-r-q,n+1-p-s+q)$&$\{\alpha_s\}$ &$0\leq q<s-r$ \\
\hline
$\mathfrak{so}(p,2n+1-p)$& $\{\alpha_r,\alpha_{r+2}\}$ & \\
 $\frak{so}(2)$&$\{\alpha_{r+1}\}$ &\\
 \hline
$\mathfrak{so}(p,2n-p)$ or $\mathfrak{so}^*(2n)$& $\{\alpha_r,\alpha_{r+2}\}$ & $r<n-3$\\
 $\frak{so}(2)$&$\{\alpha_{r+1}\}$ &\\
\hline
$\mathfrak{so}(p,q)$& $\{\alpha_2\}$ & \\
$\frak{so}(2)$& $\{\alpha_1\}$&\\
\hline
$\mathfrak{so}(p,2n-p)$& $\{\alpha_{p-2q}\}$  &$1<p-2q<n-1$\\
$\frak{u}(q,n-p+q)$& $\{\alpha_n\}$&$0\leq q$\\
\hline
$\mathfrak{so}^*(2n)$& $\{\alpha_{r}\}$&$r<n-1$\\
$\frak{u}(p,n-r-p)$& $\{\alpha_n\}$&\\
\hline
$\mathfrak{so}(n,n), \mathfrak{so}(n-1,n+1)$ or $\mathfrak{so}^*(4m+2)$& $\{\alpha_{n-3},\alpha_{n-1},\alpha_n\}$ &$n=2m+1$ \\
$\frak{so}(2)$& $\{\alpha_{n-2}\}$&\\
\hline
$\mathfrak{sp}(2n,\mathbb{R})$& $\{\alpha_r\}$&$0\leq p$ \\
$\frak{u}(p,n-r-p)$ &$\{\alpha_n\}$&\\
\hline
$\mathfrak{sp}(p,n-p)$& $\{\alpha_r\}$&  \\
$\frak{u}(p-r,n-p-r)$&$\{\alpha_n\}$&\\
\hline
\end{tabular}
\end{table}
\noindent
\begin{table}[h!]\caption{Classification of exceptional simple maximally symmetric models of 2--nondegenerate CR geometries}\label{realclasB}
\begin{tabular}{|c|c|c|c|}
\hline
$\mathfrak{g}$ & $\Sigma_1$ & $\Sigma_2$ &$\fg_{0,I}'$ \\
\hline
$\mathfrak{g}_2(2)$ & $\{\alpha_1\}$&$\{\alpha_2\}$&$\frak{so}(2)$ \\
\hline
$\mathfrak{f}_4(4)$ & $\{\alpha_2\}$&$\{\alpha_1\}$&$\frak{so}(2)$ \\
$\mathfrak{f}_4(4)$ & $\{\alpha_1,\alpha_3\}$&$\{\alpha_2\}$&$\frak{so}(2)$ \\
\hline
$\mathfrak{e}_6(6)$ & $\{\alpha_2\}$&$\{\alpha_1\}$&$\frak{so}(2)$ \\
$\mathfrak{e}_6(2)$ & $\{\alpha_6\}$&$\{\alpha_1\}$&$\frak{so}(2)\oplus \frak{u}(2,3)$\\
$\mathfrak{e}_6(-14)$ & $\{\alpha_6\}$&$\{\alpha_1\}$&$\frak{so}(2)\oplus \frak{u}(5)$ \\
$\mathfrak{e}_6(6)$ & $\{\alpha_3\}$&$\{\alpha_6\}$&$\frak{so}(2)$ \\
$\mathfrak{e}_6(2)$ & $\{\alpha_3\}$&$\{\alpha_6\}$&$\frak{so}(2)$\\
$\mathfrak{e}_6(6)$ & $\{\alpha_1,\alpha_3\}$&$\{\alpha_2\}$&$\frak{so}(2)$ \\
$\mathfrak{e}_6(6)$ & $\{\alpha_2,\alpha_4,\alpha_6\}$&$\{\alpha_3\}$&$\frak{so}(2)$ \\
$\mathfrak{e}_6(2)$ & $\{\alpha_2,\alpha_4,\alpha_6\}$&$\{\alpha_3\}$&$\frak{so}(2)$ \\
\hline
$\mathfrak{e}_7(7)$ & $\{\alpha_2\}$&$\{\alpha_1\}$&$\frak{so}(2)$\\
$\mathfrak{e}_7(-5)$ & $\{\alpha_2\}$&$\{\alpha_1\}$&$\frak{so}(2)$ \\
$\mathfrak{e}_7(-25)$ & $\{\alpha_2\}$&$\{\alpha_1\}$&$\frak{so}(2)$ \\
$\mathfrak{e}_7(7)$ & $\{\alpha_5\}$&$\{\alpha_6\}$&$\frak{so}(2)$ \\
$\mathfrak{e}_7(-5)$ & $\{\alpha_5\}$&$\{\alpha_6\}$&$\frak{so}(2)$ \\
$\mathfrak{e}_7(7)$ & $\{\alpha_6\}$&$\{\alpha_1\}$&$\frak{so}(2)\oplus \frak{so}(4,4)$ \\
$\mathfrak{e}_7(-25)$ & $\{\alpha_6\}$&$\{\alpha_1\}$&$\frak{so}(2)\oplus \frak{so}(1,7)$ \\
$\mathfrak{e}_7(7)$ & $\{\alpha_1,\alpha_3\}$&$\{\alpha_2\}$&$\frak{so}(2)$\\
$\mathfrak{e}_7(7)$ & $\{\alpha_2,\alpha_4\}$&$\{\alpha_3\}$&$\frak{so}(2)$ \\
$\mathfrak{e}_7(-5)$ & $\{\alpha_2,\alpha_4\}$&$\{\alpha_3\}$&$\frak{so}(2)$ \\
$\mathfrak{e}_7(7)$ & $\{\alpha_4,\alpha_6\}$&$\{\alpha_5\}$&$\frak{so}(2)$ \\
$\mathfrak{e}_7(-5)$ & $\{\alpha_4,\alpha_6\}$&$\{\alpha_5\}$&$\frak{so}(2)$ \\
$\mathfrak{e}_7(7)$ & $\{\alpha_3,\alpha_5,\alpha_7\}$&$\{\alpha_4\}$&$\frak{so}(2)$ \\
\hline
$\mathfrak{e}_8(8)$ & $\{\alpha_2\}$&$\{\alpha_1\}$&$\frak{so}(2)$\\
$\mathfrak{e}_8(-24)$ & $\{\alpha_2\}$&$\{\alpha_1\}$&$\frak{so}(2)$ \\
$\mathfrak{e}_8(8)$ & $\{\alpha_6\}$&$\{\alpha_7\}$&$\frak{so}(2)$ \\
$\mathfrak{e}_8(8)$ & $\{\alpha_5\}$&$\{\alpha_8\}$&$\frak{so}(2)$ \\
$\mathfrak{e}_8(8)$ & $\{\alpha_1,\alpha_3\}$&$\{\alpha_2\}$&$\frak{so}(2)$ \\
$\mathfrak{e}_8(-24)$ & $\{\alpha_1,\alpha_3\}$&$\{\alpha_2\}$&$\frak{so}(2)$ \\
$\mathfrak{e}_8(8)$ & $\{\alpha_2,\alpha_4\}$&$\{\alpha_3\}$&$\frak{so}(2)$ \\
$\mathfrak{e}_8(8)$ & $\{\alpha_3,\alpha_5\}$&$\{\alpha_4\}$&$\frak{so}(2)$ \\
$\mathfrak{e}_8(8)$ & $\{\alpha_5,\alpha_7\}$&$\{\alpha_6\}$&$\frak{so}(2)$ \\
$\mathfrak{e}_8(8)$ & $\{\alpha_4,\alpha_6,\alpha_8\}$&$\{\alpha_5\}$&$\frak{so}(2)$ \\
\hline
\end{tabular}
\end{table}
\end{thm*}

We show in Lemma \ref{absmod} and Proposition \ref{embed} that we can realize all entries of the classification as regularly 2--nondegenerate CR geometries and regularly 2--nondegenerate real submanifolds in complex space with particular second--order Levi--Tanaka algebras $(\fg_-,I,\fk)$, respectively. We call these realizations a homogeneous model $G/H$ of 2--nondegenerate CR submanifold with second--order Levi--Tanaka algebra $(\fg_-,I,\fk)$ and a standard model $\phi(\fg_-\oplus \fk)$ of 2--nondegenerate CR submanifold with second--order Levi--Tanaka algebra $(\fg_-,I,\fk)$, respectively. We compute explicit formulas for the embeddings $\phi(\fg_-\oplus \fk)$ into complex space and the corresponding defining equations for the nonexceptional hypersurface models in Section \ref{sec5}. 

The Lie algebra $\fg$ of infinitesimal CR automorphisms of the homogeneous models $G/H$ of 2--nondegenerate CR submanifold with second--order Levi--Tanaka algebras $(\fg_-,I,\fk)$ corresponding to entries of our classification provide a lower bound on the dimension of the maximally symmetric models. In Section \ref{sec4}, we prove that this is also upper bound by solving the equivalence problem for regularly 2--nondegenerate CR geometries with such second--order Levi--Tanaka algebra $(\fg_-,I,\fk)$ by proving the following theorem.

\begin{thm*}\label{abspar}
Suppose that $G,P,H,\fg_{a,b}$ satisfy the hypotheses of Theorem \ref{main} and $G/H$ is a homogeneous model of 2--nondegenerate CR submanifold with second--order Levi--Tanaka algebra $(\fg_-,I,\fk)$. Then there are $G_{0,I}$--invariant normalization conditions that provide equivalence of categories between

\begin{itemize}
\item the category of regularly 2--nondegenerate CR geometries $M$ with second--order Levi--Tanaka algebra $(\fg_-,I,\fk)$,
\item the category of $H$--fiber bundles $\ba\to M$ with a $G_{0,I}$--invariant $\fg$--valued absolute parallelism $\om$ satisfying the normalization conditions.
\end{itemize}

In particular, $dim(\fg)$ bounds the dimension of Lie algebra of infinitesimal CR automorphisms of all regularly 2--nondegenerate CR geometries with second--order Levi--Tanaka algebra $(\fg_-,I,\fk)$ and $G/H$ is the maximally symmetric model.
\end{thm*}

In fact, we construct (in Proposition \ref{parccon}) an infinitesimal $\fg_0$--structure on a regular filtered manifold with symbol $\fg_-$ and apply on it (in the first step infinitesimal version of) the Tanaka prolongation theory \cite{Ta70}. We remark that there are more refined versions of the Tanaka prolongation theory, however, we can not use them, because they require more restrictive normalization conditions that we can not satisfy, because:

\begin{itemize}
\item There are essential torsions/fundamental invariants that take value in $\fk\subset \fg_0$ and thus cannot be normalized to $0$ as in the case of usual $G_0$--structures on regular filtered manifolds with symbol $\fg_-$.
\item The fact that we choose complex linear isomorphism of $\fg_{-1}\oplus \fk$ with $T^{-1}M$ restricts the possible normalization conditions.
\item According to Lemma \ref{knf}, there is an additional reduction of the first Tanaka prolongation of $\fg_-\oplus \fg_0$ in the hypersurface case.
\end{itemize}

Let us emphasize that these are the main reasons, why we have chosen our terminology to be analogous to the Tanaka's prolongation theory. In particular, we distinguish our selves from the works of Santi \cite{Sa15} and Porter and Zelenko \cite{PZ17}. Santi does not provide a model for all second--order Levi--Tanaka algebras of CR hypersurfaces that he considers (under the name abstract core), while the formula we will obtain in Proposition \ref{embed} can be clearly used for all Levi--Tanaka algebras (although the properties of the corresponding models outside of $0$ need further investigation). Porter and Zelenko restrict themselves to CR hypersurfaces with one dimensional Levi kernel and assume that there is a particular bigrading of the complexification of $\fg$, while we will obtain this bigrading in Proposition \ref{neces} as a consequence of (semi)simplicity. Their work provides two series of maximally symmetric models with $\fg=\mathfrak{so}(p+2,q+2)$ or $\fg=\mathfrak{so}^*(2n+4)$ that are simple, however we will obtain in Section \ref{sec5} explicit defining equations for these models. The case $\fg=\frak{so}(2,3)\cong \frak{sp}(4,\mathbb{R})$ corresponds to the case of uniformly 2--nondegenerate CR hypersurfaces in $\mathbb{C}^3$ and we discuss our solution of the equivalence problem for this case in Section \ref{exam}, in detail.

\section{Real submanifolds in $\mathbb{C}^N$ and 2--nondegenerate CR submanifolds}\label{sec2}

On a real submanifold $M\subset \mathbb{C}^N$ there is the maximal complex subspace $T^{-1}M=TM\cap i(TM)$ with complex structure $\mathcal{I}$ induced by multiplication by $i$. There is an open dense subset, where $T^{-1}M$ defines a filtration $T^{-1}M\subset \dots \subset T^{\mu}M\subset TM$, where $T^{\mu}M$ is integrable. Therefore, we can consider leaves of the corresponding foliation and assume that $(M,T^iM)$ is a filtered manifold with a complex structure $\mathcal{I}_x$ on each $T^{-1}_xM$ and say that the triple $(M,T^iM,\mathcal{I})$ is a CR submanifold. On the other hand, the assumption that the filtration is regular is not a generic assumption and we will not assume it in this section.

Further, there is an open dense subset of $M$, where the Levi kernel $\mathcal{K}$ is a distribution. We observe that $[\mathcal{K},\mathcal{K}]\subset \mathcal{K}$ and $[\mathcal{K},T^{-1}M]\subset T^{-1}M$ hold as a consequence of the definition of Levi kernel. Therefore, $[\mathcal{K},T^{-1}M/\mathcal{K}]\subset T^{-1}M/\mathcal{K}$ and we see that \begin{align*}
[A,fX]+\mathcal{I}[A,f\mathcal{I}(X)]&=f([A,X]+\mathcal{I}[A,\mathcal{I}(X)])+(A.f)X+\mathcal{I}((A.f)\mathcal{I}(X))\\
&=f([A,X]+\mathcal{I}[A,\mathcal{I}(X)])
\end{align*}
 holds for sections $A$ of $\mathcal{K}$ and $X$ of $T^{-1}M$ and smooth function $f$ and its directional derivative $A.f$. Therefore, the complex anti--linear part of the Lie bracket of these sections is algebraic and the map $\iota_x: \mathcal{K}_x\to \frak{gl}(T^{-1}_xM/\mathcal{K}_x)$ is well--defined. On 2--nondegenerate CR manifold, we identify $\mathcal{K}_x$ with its image in $\frak{gl}(T^{-1}_xM/\mathcal{K}_x)$.

Now, at each point of a 2--nondegenerate CR submanifold $(M,T^iM,\mathcal{I})$, we have the Lie algebra $\fm_x:=(gr(T_xM),\mathcal{L}_x)$, the complex structure $\mathcal{I}_x$ and subspace $\mathcal{K}_x\subset \frak{gl}(T^{-1}_xM/\mathcal{K}_x)$. Let us summarize the properties of these objects:

\begin{lem*}\label{LTa}
\begin{enumerate}
\item $(\fm_x,\mathcal{I}_x)$ is a Levi--Tanaka algebra, i.e., $$\mathcal{L}_x(\mathcal{I}_x(X),\mathcal{I}_x(Y))=\mathcal{L}_x(X,Y)$$ for all $X,Y\in T^{-1}_xM.$
\item $(\fm_x/\mathcal{K}_x,\mathcal{I}_x|_{T^{-1}_xM/\mathcal{K}_x})$ is a nondegenerate Levi--Tanaka algebra.
\item The action of $\mathcal{K}_x$ on $gr_{-1}(T_xM)$ extends trivially on $gr_{-2}(T_xM)$.
\item The action of $\mathcal{I}_x|_{\mathcal{K}_x}$ is given by the composition the endomorphisms with $\mathcal{I}_x|_{T^{-1}_xM/\mathcal{K}_x}$
\end{enumerate}
\end{lem*}
\begin{proof}
The complexification $T^{-1}M\otimes \mathbb{C}$ of $T^{-1}M$ decomposes  according to the eigenvalues of the complex structure $\mathcal{I}$ to $T^{-1,10}M\oplus T^{-1,01}$. Since the CR submanifold is induced by a real submanifold $M$ of complex space $\mathbb{C}^N$, it satisfies a (formal) integrability condition $[T^{-1,10}M,T^{-1,10}M]\subset T^{-1,10}M$. This directly implies that $\mathcal{L}_x(\mathcal{I}_x(X),\mathcal{I}_x(Y))=\mathcal{L}_x(X,Y)$ for all $X,Y\in T^{-1}_xM,$ i.e.,  $(gr(T_xM),\mathcal{L}_x,\mathcal{I}_x)$ is a Levi--Tanaka algebra.

Since elements of the Levi kernel belong to an ideal in  $(gr(T_xM),\mathcal{L}_x)$ and the complex structure preserves the Levi kernel, the second claim holds.

Further, $gr_{-2}(T_xM)\otimes \mathbb{C}=\mathcal{L}_x(T^{-1,10}M,T^{-1,01}M)$ and action of $\mathcal{K}_x$ interchanges $T^{-1,10}M$ and $T^{-1,01}M$. Therefore, $\mathcal{L}_x(A(T^{-1,10}M),T^{-1,01}M)=0$ for all $A\in \mathcal{K}_x$ and the third claim holds.

If we insert $\mathcal{I}(A)$ instead of $A$ into definition of $\iota_x$ use the fact that $(X.\mathcal{I})A_x\in \mathcal{K}_x,(\mathcal{I}(X).\mathcal{I})A_x\in \mathcal{K}_x$, then we obtain that $[\mathcal{I}(A),X]+\mathcal{I}[\mathcal{I}(A),\mathcal{I}(X)]=\mathcal{I}([A,X]+\mathcal{I}[A,\mathcal{I}(X)])$, i.e., the last claim holds.
\end{proof}

We see from Lemma \ref{LTa} that the following is the appropriate generalization of the Levi--Tanaka algebra to the setting of 2--nondegenerate CR submanifolds, which is well--defined almost everywhere.

\begin{def*}
We say that the triple $(\fm_x/\mathcal{K}_x,\mathcal{I}_x|_{T^{-1}_xM/\mathcal{K}_x},\mathcal{K}_x\subset \frak{gl}(T^{-1}_xM/\mathcal{K}_x))$ is the second--order Levi--Tanaka algebra of the CR submanifold $(M,T^{i}M,\mathcal{I})$ at $x$.

We say that the second--order Levi--Tanaka algebra is regular, if all elements of $\mathcal{K}_x$ extend to grading preserving derivations of $\fm_x/\mathcal{K}_x$ and $[[\mathcal{K}_x,\mathcal{K}_x],\mathcal{K}_x]\subset \mathcal{K}_x$ holds for the Lie bracket of these derivations.
\end{def*}

In the next Sections, we look on Levi--Tanaka algebras on homogeneous 2--nondegenerate CR submanifolds and show that we can restrict ourselves to regular second--order Levi--Tanaka algebras.

\section{Maximally symmetric 2--nondegenerate CR geometries}

\subsection{Necessary conditions}\label{sec31}

A homogeneous CR submanifold is always regular, however, a homogeneous 2--nondegenerate CR submanifold satisfies a priory just the condition (1) of the Definition \ref{def2dg} and does not have to satisfy the conditions (2) and (3),i .e., does not have to be regularly 2--nondegenerate.

Suppose $G$ is a Lie group of CR automorphisms of a 2--nondegenerate CR geometry on a regular filtered manifold $(M,T^{i}M)$ with Levi-Tanaka algebra $(\fm,I)$ and second--order Levi--Tanaka algebra $(\fg_-,I,\fk)$ acting effectively and transitively on $M$. If we identify $T_xM$ with the vector space $\fg/\fh$, where $\fh$ is the Lie algebra of stabilizer $H$ of $x$, then there is a filtration $\fg^i$ of $\fg$ such that $gr_i(T_xM)=(\fg^i/\fh)/(\fg^{i+1}/\fh)$ holds.

\begin{lem*}\label{autgrad}
Let $\fg^0$ be the preimage of $\fk$ in $\fg$ and define $\fg^j$ to be as the set of all elements $X\in \fh$ such that $[X,\fg^{-1}]\subset \fg^{j-1}$. Then
\begin{enumerate}
\item $[\fg^j,\fg^k]\subset \fg^{j+k}$ for all $j,k$,
\item $\fg^\ell=0$ for $\ell$ large enough,
\item there is associated graded Lie algebra $gr(\fg)$ with graded Lie subalgebra $gr(\fh)$,
\item there is a homogeneous 2--nondegenerate CR geometry with second--order Levi--Tanaka algebra $(\fg_-,I,\fk)$ with Lie algebra of infinitesimal CR automorphism containing $gr(\fg)$ and the grading element of $gr(\fg)$.
\end{enumerate}
\end{lem*}
\begin{proof}
If $i,j<0$, then Claim (1) holds by definition of the filtered manifold. If $i<0$ and $j= 0$, then Claim (1) follows from the definition of the Levi kernel. The remaining cases of Claim (1) follow directly from the Jacobi identity. 

The Claim (2) is an obvious consequence of Claim (1) and the effectivity of the action of $G$.

The Claim (3) is natural consequence of Claim (1). Moreover, the grading element of $gr(\fg)$ is a derivation annihilating $gr_0(\fg)$. Moreover, $gr_0(\fh)$ is intersection of fixed point set of conjugation by $I$ in $\frak{gl}(gr_{-1}(\fg))$ with $gr_0(\fg)$, thus there is a homogeneous spaces $\exp(\fg_-)\times \exp(gr_0(\fg))/\exp(gr_0(\fh))$ with an invariant 2--nondegenerate CR geometry, where the grading element is an infinitesimal CR automorphism. This proves the Claim (4).
\end{proof}

This Lemma immediately implies the necessity of regular 2--nondegeneracy of the second--order Levi--Tanaka algebra in the (semi)simple case and one of the implications in the Theorem \ref{main}.

\begin{prop*}\label{neces}
Suppose $\fg$ is semisimple and $dim(\fg)$ is maximal among all (homogeneous) 2--nondegenerate CR geometries with second--order Levi--Tanaka algebra $(\fg_-,I,\fk)$. Then $\fh$ is a subalgebra of a parabolic subalgebra $\fp$ of $\fg$ and there is bigrading $\fg_{a,b}$ of the complexification of $\fg$ such that
\begin{enumerate}
\item $(\fg_0,\fg_{0,I})$ is a (pseudo)--Hermitian symmetric pair,
\item $\fg_a\otimes \mathbb{C}=\sum_b \fg_{a,b}$ holds for the grading $\fg_i$ of $\fg$ corresponding $P$,
\item $\fg_{0,I}\otimes \mathbb{C}=\fg_{0,0},$
\item $\fg_{0}\otimes \mathbb{C}=\fg_{0,-1}\oplus \fg_{0,0}\oplus \fg_{0,1}$,
\item $\fg_{-1}\otimes \mathbb{C}=\fg_{-1,-1}\oplus \fg_{-1,0}$,
\item $\fg_{-2}\otimes \mathbb{C}=\fg_{-2,-1}.$
\end{enumerate}
In particular, the second--order Levi--Tanaka algebra $(\fg_-,I,\fk)$ is regular.
\end{prop*}
\begin{proof}
A consequence of Claim (4) of Lemma \ref{autgrad} is that if the grading element $E_1$ is not element of $gr(\fg)$, then $dim(\fg)$ is not maximal among all CR manifold with the given second--order Levi--Tanaka algebra. Clearly, the semisimple part of the adjoint action of the grading element $E_1\in gr(\fg)$ is a grading element and if $\fg$ is semisimple, then $E_1\in \fg$ and $\fg=gr(\fg)$. Thus there is a corresponding parabolic subalgebra $\fp$ of $\fg$ containing $\fh$.

Now, we consider the complexified situation. Both Lie algebras $(\fg_{0}\cap \fh)\otimes \mathbb{C}$ and $\fg_{0}\otimes \mathbb{C}$ are reductive. Moreover, we can choose $(\fg_{0}\cap \fh)\otimes \mathbb{C}$--invariant complements of $(\fg_{0}\cap \fh)\otimes \mathbb{C}$ in $\fg_{0}\otimes \mathbb{C}$ consisting of nilpotent elements to represent the $\pm i$--eigenspaces of $I$ in the complexification of $\fk$. Therefore, $(\fg_{0}\cap \fh)\otimes \mathbb{C}$ contains a Cartan subalgebra of $\fg$ (and the grading element) and $I$ acts by eigenvalues $\pm i$ on the corresponding root spaces. Since bracket of root spaces is a root space, the action of elements of $\fk$ swap $\pm i$--eigenspaces of $I$ in $\fg_{-1}\otimes \mathbb{C}$ and the bracket (in $\fg$) of elements $\fk$ preserves $\pm i$--eigenspaces of $I$ in $\fg_{-1}\otimes \mathbb{C}$. Therefore, the functional on the set of roots in $(\fg_{0}\oplus \fg_{-1})\otimes \mathbb{C}$ defined 
\begin{itemize}
\item as $\pm 1$ on roots in $\mp i$--eigenspaces of $\mathcal{I}_x$ in $\fg_{-1}\otimes \mathbb{C}$,
\item as $\pm 2$ on roots in $\mp i$--eigenspaces of $\mathcal{I}_x$ in $\fk\otimes \mathbb{C}$,
\item as $0$ on roots in $(\fg_{0}\cap \fh)\otimes \mathbb{C}$.
\end{itemize}
is compatible with the Lie bracket $(\fg_{0}\otimes \mathbb{C})\otimes ((\fg_{0}\oplus \fg_{-1})\otimes \mathbb{C})\to (\fg_{0}\oplus \fg_{-1})\otimes \mathbb{C}$. Since $(\fg_{0}\oplus \fg_{-1})\otimes \mathbb{C}$ contains all simple negative roots, there is element $\tilde E$ in the Cartan subalgebra of $\fg$ that realizes this functional.

Let us consider bigrading of $\fg$ given by grading elements $E_1$ and $E_2=\frac12(E_1+\tilde E)$. It is simple computation to show that this bigrading has the claimed properties (2)--(6). Finally, $[I(X),I(Y)]=\frac14([iX^{10},-iY^{01}]+[-iX^{01},iY^{10}])=\frac14([X^{10},Y^{01}]+[X^{01},Y^{10}])=[X,Y]\in \fg_0$ holds for the decompositions of $X=X^{10}+X^{01},Y=Y^{10}+Y^{01}\in \fk\subset \fg_{0,-1}\oplus \fg_{0,1}$. We see that $(\fg_0,\fg_{0,I}=(\fg_0\cap \fh))$ is a symmetric pair with complex structure $I$ that is twisted in the terminology of \cite{Be00} and \cite[Proposition V.2.2]{Be00} ensures that it is a (pseudo)--Hermitian symmetric pair. Therefore, the second--order Levi--Tanaka algebra  $(\fm/\fk,I|_{\fm_{-1}/\fk},\fk)$ is regular.
\end{proof}

We can observe that the necessary conditions (1)--(6) from Proposition \ref{neces} also hold in the case of semisimple maximally symmetric homogeneous 2--degenerate CR geometries, because the kernel of the map $\iota_x: \mathcal{K}_x\to (T^{-1}_xM/\mathcal{K}_x)$ is contained in $\fg^1\subset \fp$ and plays no role in the construction of the bigrading. Therefore, we can use the results of Freeman \cite{Fr74,Fr77} to show that there can not be semisimple maximally symmetric homogeneous CR geometries that are nondegenerate of order $\ell\geq 3$.

\begin{thm*}\label{neces3}
There is no semisimple maximally symmetric homogeneous CR geometry that is nondegenerate of order $\ell\geq 3$ .
\end{thm*}
\begin{proof}
The result of Freeman \cite{Fr74,Fr77} is that in order to have finitely dimensional CR automorphism group, we need to assume the CR geometry is finitely nondegenerate of some order $\ell$. If $\ell\geq 3$, the part of the grading component $\fg_1$ in $\fg/\fh$ that is in kernel of $\iota_x$ has to map $\fg_{-1,0}$ into $\fg_{0,-1}$ after complexification, i.e., there is $\fg_{1,-1}\subset \fg_1\otimes \mathbb{C}$ and by duality $\fg_{-1,1}\subset \fg_{-1}\otimes \mathbb{C}$, which is contradiction with the necessary condition $\fg_{-1}\otimes \mathbb{C}=\fg_{-1,-1}\oplus \fg_{-1,0}$
\end{proof}

 \subsection{Classification}\label{sec33}
 
We classify all bigradings of (complex) simple Lie algebras satisfying the conditions of Proposition \ref{neces} and all real forms that provide (after proving Theorem \ref{main}) the maximally symmetric models of 2--nondegenerate CR manifolds. Firstly, let us check that such bigradings define homogeneous models of 2--nondegenerate CR geometries. In Section \ref{sec34}, we show that these CR geometries can be (locally) embed into complex space.

\begin{lem*}\label{absmod}
Let $\fg$ be a real form of a bigraded semisimple complex Lie algebra $\fg_{\mathbb{C}}=\oplus_{a,b}\fg_{a,b}$ with grading elements $E_1,E_2$ . Suppose that the grading element $E_1$ preserves $\fg$ and denote $\fg_i$ the corresponding grading. If $\fg_0\otimes \mathbb{C}=\fg_{0,-1}\oplus \fg_{0,0}\oplus \fg_{0,1}$, $\fg_{-1}\otimes \mathbb{C}= \fg_{-1,0}\oplus  \fg_{-1,-1}, \fg_{-2}\otimes \mathbb{C}=\fg_{-2,-1}$ and $(\fg_0,\fg_{0,I}=\fg_0\cap \fg_{0,0})$ is a (pseudo)--Hermitian pair, then the homogeneous spaces $G/H$ and $(\exp(\fg_-)\rtimes G_0)/G_{0,I}$ carry an invariant  2--nondegenerate CR geometry with second--order Levi--Tanaka algebra $(\fg_-,I,\fk=\fg_0\cap(\fg_{0,-1}\oplus \fg_{0,1}))$, where $I=iad(E_1-2E_2)|_{\fg_{-1}}$.
\end{lem*}
\begin{proof}
If $(\fg_0,\fg_{0,I})$ is a (pseudo)--Hermitian symmetric pair, then $-iad(E_2)$ preserves $\fg_0$ and thus it is a derivation of $\fg$. It follows from the classification below that $\fh$ acts on $\fg_{-1}$ by a complex representation and thus $I=iad(E_1-2E_2)|_{\fg_{-1}}$ is well--defined.  Since $\fg_{1}\subset \fg_{1,0}\oplus  \fg_{1,1}$ and $[iad(E_1-2E_2)\fg_{-1,0},\fg_{1,1}]=-i\fg_{0,1}=-iad(E_2)\fg_{0,1}$ and  $[iad(E_1-2E_2)\fg_{-1,-1},\fg_{1,0}]=i\fg_{0,-1}=-iad(E_2)\fg_{0,-1}$ we see that elements $\fg_1$ preserve the complex structure on $\fg_{-1}\oplus\fk$ induced by $I$. Therefore, there is $H=G_{0,I}\rtimes \exp(\fp_+)$--invariant complex structure on a distribution on $G/H$ given by $\fg_{-1}\oplus \fk$ and $(\exp(\fg_-)\rtimes G_0)/G_{0,I}$ is an open submanifold of $G/(G_{0,I}\rtimes \exp(\fp_+))$. The remaining claims are obtained by a simple computation.
\end{proof}

There are few obvious conditions for the sets of simple roots $\Sigma_1$ and $\Sigma_2$ defining the bigradings from Proposition \ref{neces} that are visible due to the fact that for any path in the Dynkin diagram of $\fg\otimes \mathbb{C}$, there is a roots space given by the sum of the simple roots on the path:

\begin{itemize}
\item Along any path in the Dynkin diagram if we ignore simple roots outside of $\Sigma_1\cup \Sigma_2$, then simple roots inside and outside of $\Sigma_2$ has to alternate, otherwise there would be a root space not satisfying the conditions from Proposition \ref{neces}. Moreover, if there would be simple root in both $\Sigma_1$ and $\Sigma_2$, then $\Sigma_2\subset \Sigma_1$ clearly holds and we are in the situation of Theorem \ref{nondegclas}. Therefore, the bigradings are given by a choice of two distinct sets of simple roots $\Sigma_1$ and $\Sigma_2$, which are alternating along all paths in the Dynkin diagram (if we ignore simple roots outside of $\Sigma_1\cup \Sigma_2$).
\item If one removes the nodes of $\Sigma_1$ or $\Sigma_2$ from the Dynkin diagram of $\fg$, then $\Sigma_2$ or $\Sigma_1$ define a $|1|$--gradings of the Lie algebras given by the remaining nodes, respectively. In particular, $\fg_{-1,-1}$ is the $-2$--grading component of grading given by $\Sigma_1\cup \Sigma_2$.
\begin{cor*}
The bigrading of $\fg$ given by the sets $\Sigma_1\cup \Sigma_2$ and $\Sigma_2$ satisfies the properties of the Theorem \ref{nondegclas}.
\end{cor*}
\item $|\Sigma_2|=1$, because we would get contradiction with $\fg_{-1}\otimes \mathbb{C}= \fg_{-1,0}\oplus  \fg_{-1,-1}$ if $|\Sigma_2|\neq 1$.
\item $|\Sigma_1|$ is less than number of connected components of  Dynkin diagram of $\fg$ with $\Sigma_2$ removed.
\end{itemize}

These conditions allow us to classify all bigradings with the properties from Proposition \ref{neces} in an algorithmic fashion. 

\begin{enumerate}
\item We start with simple Lie algebra $\fg$ and pick one simple root to be in $\Sigma_2$.
\item Choose $\Sigma_1$ by picking at most one root in each connected components of the Dynkin diagram of $\fg$ with $\Sigma_2$ removed corresponding to a $|1|$--grading. Check that $\Sigma_2$ corresponds to $|1|$--grading of Dynkin diagram of $\fg$ with $\Sigma_1$ nodes removed and that $\Sigma_1$ defines at least $2$ grading of $\fg$. Equivalently, one can cross check that $(\fg,\Sigma_1\cup \Sigma_2)$ appears in the classification of \cite{AMT06}.
\item Check that $-3$--grading component of $(\fg,\Sigma_1\cup \Sigma_2)$ is $\fg_{-2,-1}$ and check that $-4$--grading component of $(\fg,\Sigma_1\cup \Sigma_2)$ is $\fg_{-3,-1}$. 
\end{enumerate}

\begin{lem*}
The set of all bigradings $\fg_{a,b}$ of complex simple Lie algebras $\mathfrak{g}_{\mathbb{C}}$ with the properties from Proposition \ref{neces} corresponds to the entries of the Table \ref{class}.
\noindent
\begin{table}\caption{Admissible bigradings}\label{class}
\begin{tabular}{|c|c|c|c|c|}
\hline
$\mathfrak{g}_{\mathbb{C}}$ & $\Sigma_1$ & $\Sigma_2$& restrictions\\
\hline
$\mathfrak{sl}(n+1,\mathbb{C})$&$\{\alpha_r,\alpha_s\}$&$\{\alpha_t\}$  & $r<t<s$ \\
\hline
$\mathfrak{so}(2n+1,\mathbb{C})$& $\{\alpha_r,\alpha_{r+2}\}$&$\{\alpha_{r+1}\}$&  \\
$\mathfrak{so}(2n+1,\mathbb{C})$& $\{\alpha_2\}$& $\{\alpha_1\}$&  \\
\hline
$\mathfrak{sp}(2n,\mathbb{C})$& $\{\alpha_r\}$&$\{\alpha_n\}$&\\
\hline
$\mathfrak{so}(2n,\mathbb{C})$& $\{\alpha_2\}$& $\{\alpha_1\}$& \\
$\mathfrak{so}(2n,\mathbb{C})$& $\{\alpha_r\}$& $\{\alpha_n\}$&  $1<r<n-1$\\
$\mathfrak{so}(2n,\mathbb{C})$& $\{\alpha_r,\alpha_{r+2}\}$&$\{\alpha_{r+1}\}$&  $r< n-3$\\
$\mathfrak{so}(2n,\mathbb{C})$& $\{\alpha_{n-3},\alpha_{n-1},\alpha_n\}$& $\{\alpha_{n-2}\}$& \\
\hline
$\mathfrak{g}_2(\mathbb{C})$ & $\{\alpha_1\}$&$\{\alpha_2\}$&\\
\hline
$\mathfrak{f}_4(\mathbb{C})$ & $\{\alpha_2\}$&$\{\alpha_1\}$&\\
$\mathfrak{f}_4(\mathbb{C})$ & $\{\alpha_1,\alpha_3\}$&$\{\alpha_2\}$&\\
\hline
$\mathfrak{e}_6(\mathbb{C})$ & $\{\alpha_2\}$&$\{\alpha_1\}$&\\
$\mathfrak{e}_6(\mathbb{C})$ & $\{\alpha_6\}$&$\{\alpha_1\}$&\\
$\mathfrak{e}_6(\mathbb{C})$ & $\{\alpha_3\}$&$\{\alpha_6\}$&\\
$\mathfrak{e}_6(\mathbb{C})$ & $\{\alpha_1,\alpha_3\}$&$\{\alpha_2\}$&\\
$\mathfrak{e}_6(\mathbb{C})$ & $\{\alpha_2,\alpha_4,\alpha_6\}$&$\{\alpha_3\}$&\\
\hline
$\mathfrak{e}_7(\mathbb{C})$ & $\{\alpha_2\}$&$\{\alpha_1\}$&\\
$\mathfrak{e}_7(\mathbb{C})$ & $\{\alpha_5\}$&$\{\alpha_6\}$&\\
$\mathfrak{e}_7(\mathbb{C})$ & $\{\alpha_6\}$&$\{\alpha_1\}$&\\
$\mathfrak{e}_7(\mathbb{C})$ & $\{\alpha_4\}$&$\{\alpha_6\}$&\\
$\mathfrak{e}_7(\mathbb{C})$ & $\{\alpha_1,\alpha_3\}$&$\{\alpha_2\}$&\\
$\mathfrak{e}_7(\mathbb{C})$ & $\{\alpha_2,\alpha_4\}$&$\{\alpha_3\}$&\\
$\mathfrak{e}_7(\mathbb{C})$ & $\{\alpha_4,\alpha_6\}$&$\{\alpha_5\}$&\\
$\mathfrak{e}_7(\mathbb{C})$ & $\{\alpha_3,\alpha_5,\alpha_7\}$&$\{\alpha_4\}$&\\
\hline
$\mathfrak{e}_8(\mathbb{C})$ & $\{\alpha_2\}$&$\{\alpha_1\}$&\\
$\mathfrak{e}_8(\mathbb{C})$ & $\{\alpha_6\}$&$\{\alpha_7\}$&\\
$\mathfrak{e}_8(\mathbb{C})$ & $\{\alpha_5\}$&$\{\alpha_8\}$&\\
$\mathfrak{e}_8(\mathbb{C})$ & $\{\alpha_1,\alpha_3\}$&$\{\alpha_2\}$&\\
$\mathfrak{e}_8(\mathbb{C})$ & $\{\alpha_2,\alpha_4\}$&$\{\alpha_3\}$&\\
$\mathfrak{e}_8(\mathbb{C})$ & $\{\alpha_3,\alpha_5\}$&$\{\alpha_4\}$&\\
$\mathfrak{e}_8(\mathbb{C})$ & $\{\alpha_5,\alpha_7\}$&$\{\alpha_6\}$&\\
$\mathfrak{e}_8(\mathbb{C})$ & $\{\alpha_4,\alpha_6,\alpha_8\}$&$\{\alpha_5\}$&\\
\hline
\end{tabular}
\end{table}
\end{lem*}
\begin{proof}
Type $A$: We need $|\Sigma_1|=2$ to be satisfied in order to get two grading, otherwise, there are no restrictions on $\Sigma_1,\Sigma_2$.

Type $B$: If $|\Sigma_1|=1$, then the nodes of $\Sigma_1$ and $\Sigma_2$ has to be next to each other. If $1\notin \Sigma_2$, then the root space of $\alpha_{i-1}+2\alpha_{i}+2\alpha_{i+1}+\dots$ is in $\fg_{-2,-2}$. If $1\in \Sigma_2$, then $\alpha_1+2\alpha_2+\dots$ is the unique root with root space in $-3$--grading component of $(\fg,\Sigma_1\cup \Sigma_2)$.

Type $B$: If $|\Sigma_1|=2$, then if the nodes of $\Sigma_1$ and $\Sigma_2$ are not next to each other, then the root space of $\alpha_{i-1}+2\alpha_{i}+2\alpha_{i+1}+\dots$ is in $\fg_{-2,-2}$. Otherwise, $\alpha_i+2\alpha_{i+1}+\dots$, $\dots+\alpha_{i-1}+\alpha_i+\alpha_{i+1}+\dots$ are the only root spaces in $-3$--grading component of $(\fg,\Sigma_1\cup \Sigma_2)$ and $\dots+\alpha_{i-1}+2\alpha_i+2\alpha_{i+1}+\dots$  are the only root spaces in $-4$--grading component of $(\fg,\Sigma_1\cup \Sigma_2)$.

Type $C$: If $|\Sigma_1|=1$, then $\Sigma_2=\{l\}$ and $\dots+2\alpha_i+\dots+\alpha_l$ are the only root spaces in  $-3$--grading component of $(\fg,\Sigma_1\cup \Sigma_2)$.

Type $C$: If $|\Sigma_1|=2$, then $\{l\}\subset \Sigma_1$ and $2\alpha_i+\dots+\alpha_l$ has root space in $\fg_{-1,-2}$.

Type $D$:  If $|\Sigma_1|=1$, then either $\Sigma_1$ and $\Sigma_2$ has to be next to each other or $\Sigma_2=\{l\}$ (up to outer automorphism of the Dynkin diagram). The first case coincides with the type $B$. In the second case, then $\dots+2\alpha_2+\dots+\alpha_l$ are the only roots with root space in $-3$--grading component of $(\fg,\Sigma_1\cup \Sigma_2)$.

Type $D$:  If $|\Sigma_1|=2$, then if $\Sigma_1$ and $\Sigma_2$ are not next to each other then $2\alpha_i+\dots$ is either in $\fg_{-2,-2}$ or $\fg_{-1,-2}$. If $X_2\neq \{l-2\}$, then we are in the case that coincides with the $B$ case.  If $\Sigma_2=\{l-2\}$, then $\alpha_{l-3}+2\alpha_{l-2}+\dots$ has root space in $\fg_{-2,-2}$.

Type $D$:  If $|\Sigma_1|=3$,  then $\Sigma_2=\{l-2\}$ and $\{l-1,l\}\subset \Sigma_1$. If $l-3\notin\Sigma_1$, then $\alpha_{l-3}+2\alpha_{l-2}+\dots$ has root space in $\fg_{-2,-2}$. Otherwise, all roots containing two $\alpha_{l-2}$ are in $\fg_{-3,-2}$ and $\fg_{-4,-2}$.

Types $E,F,G$: The entries of the tables can be obtained by going through the finite number of all possibilities.
\end{proof}

Thus it remains to go trough the list of real forms $\fg$ of the Lie algebras from the Table \ref{class} that admit the grading given by $\Sigma_1$ and check the classification of the (pseudo)--Hermitian symmetric spaces, c.f. \cite{odk9}, whether $(\fg_0',\fg_{0,I}')$ is (pseudo--)Herminitan symmetric pair, where $(\fg_0',\fg_{0,I}')$ is the effective quotient of the pair $(\fg_0,\fg_{0,I})$. This provides us the Tables \ref{realclasA}, \ref{realclasB}.

\subsection{Embedding of the models}\label{sec34}

Let us consider the homogeneous spaces $G/H$ and $(\exp(\fg_-)\rtimes G_0)/G_{0,I}$ from the Lemma \ref{absmod} corresponding to entries from our classification in Tables \ref{realclasA}, \ref{realclasB}. We show that we can (locally) embed them into a complex space. We call this embedding the standard model of 2--nondegenerate CR submanifold with regular second--order Levi--Tanaka algebra $(\fg_-,I,\fk)$.

\begin{prop*}\label{embed}
Let $(\fg_-,I,\fk)$ be a regular second--order Levi--Tanaka algebra with homogeneous model $G/H$ corresponding to an entry in Table \ref{realclasA} or \ref{realclasB}. Then there is a realization of the real form $\fg$ of the complexification $\fg\otimes \mathbb{C}$ such that $\fn:=(\fg_{-\mu}\oplus \dots \oplus \fg_{-2})\otimes \mathbb{C}\oplus \fg_{-1,-1}\oplus \fg_{0,-1}= \fg\otimes \mathbb{C}/\fp_{\Sigma_2}$ and $\fg\cap \fp_{\Sigma_2}=\fh$, where $\fp_{\Sigma_2}$ is a parabolic subalgebra of $\fg\otimes \mathbb{C}$ determined by set of simple roots $\Sigma_2$. In particular, $G$ has (up to a covering) an open orbit with stabilizer $H$ in the complex flag variety of type $(\fg\otimes \mathbb{C},\fp_{\Sigma_2})$.

Moreover, let $\phi: \fg_-\oplus\fk\to \fn$ be given by
\begin{align*}
&\phi(X+Y):=\exp^{-1}(\exp(X)\exp(\frac12(Y-iI(Y)))\exp(-\frac12(X_{-1}+iI(X_{-1}))))\\
&=\frac12(X_{-1}-iI(X_{-1}))+\frac14[X_{-1}+iI(X_{-1}),Y-iI(Y)]+\frac12(Y-iI(Y)))\\
&+X_{-2}+i\frac14[I(X_{-1}),X_{-1}]+\frac18[X_{-1}+iI(X_{-1}),[X_{-1}+iI(X_{-1}),Y-iI(Y)]]\\
&\vdots\\
&+\phi_1(X)_{\mu}+\frac{1}{|\mu|!}ad(\frac12(X_{-1}+iI(X_{-1})))^{|\mu|}(\frac12(Y-iI(Y)))\\
&+f(\phi_1(X)_{\mu+1}+\dots+\phi_1(X)_{-1},\sum_{i=0}^{\mu-1}ad(\frac12(X_{-1}+iI(X_{-1})))^i(\frac12(Y-iI(Y))))_{\mu}
\end{align*}
for $X=X_{\mu}+\dots+X_{-1}\in \fg_{-\mu}\oplus \dots \oplus \fg_{-1}$ and $Y\in \fk$, where $f(Y,Z)_{-i}$ denotes the component of $$f(Y,Z):=\sum_{n=1}^{\infty}\frac{(-1)^{n}}{n+1}\sum_{r_i+s_i>0}\frac{ad(Y)^{r_1}ad(Z)^{s_1}\dots ad(Y)^{r_n}ad(Z)^{s_n}(Y)}{(1+\sum_{i=1}^nr_i)\prod_{i=1}^n r_i!s_i!}$$
 in $\fg_{-i}\otimes \mathbb{C}$. Then the submanifold $\phi(\fg_-\oplus\fk)$ of the complex space $\fn$ is a regularly 2--nondegenerate CR submanifold with second--order Levi--Tanaka algebra $(\fg_-,I,\fk)$ and Lie algebra of infinitesimal CR automorphisms $\fg$.
\end{prop*}
\begin{proof}
If we consider the complexified symmetric space $G_{0,\mathbb{C}}/G_{0,0}$ corresponding to the simple complex Hermitian symmetric pair $(\fg_0\otimes \mathbb{C},\fg_{0,0})$, then we have complementary (at $e$) subgroups $\exp(\fg_{0,1})$ and $G_{0,0}\rtimes \exp(\fg_{0,-1})$ of $G_{0,\mathbb{C}},$ i.e., the abelian complex Lie group $\exp(\fg_{0,1})$ has open orbit in the complex manifold $G_{0,\mathbb{C}}/G_{0,0}\rtimes \exp(\fg_{0,-1})$. At the same time $\fk\subset \fg_{0,-1}\oplus \fg_{0,1}$ defines an open submanifolds $\exp(\fk)$ in  the complex manifold $G_{0,\mathbb{C}}/G_{0,0}\rtimes \exp(\fg_{0,-1})$. If $X$ in a vector space with a complex structure $I$, then $X-iI(X)$ is in $-i$--eigenspace of $I$ on the complexification and $X+iI(X)$ is in $i$--eigenspace of $I$ on the complexification. Therefore, $\exp(\frac12(Y-iI(Y)))\mapsto Y$ for $Y$ in some neighborhood of $0$ in $\fk$ provide coordinates of $\exp(\fk)$ (different than the logarithmic coordinates). Therefore, there is a local diffeomorphism between a neighborhood of $0$ in  $G/H$ and its image in the open submanifold $\exp(\fn)$ of the complex manifold $(\exp(\fg_-\otimes \mathbb{C})\rtimes G_{0,\mathbb{C}})/(\exp(\fg_{-1,-1})\rtimes G_{0,0}\rtimes \exp(\fg_{0,-1}))$. Choosing a point in the open submanifold $\exp(\fn)$, which is the so--called big cell of the complex flag variety of type $(\fg\otimes \mathbb{C},\fp_{\Sigma_2})$ provides (by conjugation) the real form $\fg$ with the claimed properties.

Thus it remains to show that $\phi$ is the (analytic extension) of this local diffeomorphism in some coordinates. Clearly, $\exp(X)\exp(\frac12(Y-iI(Y)))\mapsto X+Y$ provides coordinates of the open submanifold $\exp(\fg_-)\exp(\fk)\subset G/H$. Further, we use the logarithmic coordinates on $\exp(\fn)\subset (\exp(\fg_-\otimes \mathbb{C})\rtimes G_{0,\mathbb{C}})/(\exp(\fg_{-1,-1})\rtimes G_{0,0}\rtimes \exp(\fg_{0,-1}))$. Since $\fg_{0,1}$ is subalgebra of $\fn$, there is an element of $Z\in \fg_{-1,-1}$ such that $\exp(X)\exp(\frac12(Y-iI(Y)))\exp(Z)\in \exp(\fn)$. It is clear from the Baker--Campbell--Hausdorff formula for $\exp^{-1}(\exp(X)\exp(\frac12(Y-iI(Y)))\exp(Z))$ that $Z$ and $\phi$ have the claimed form.

Let us express the infinitesimal CR automorphism $A\in \fg$ in the coordinates of the embedding $\phi: \fg_-\oplus \fk\to \fn$. If we identify $T_{y}\exp(\fn)\cong \fn$ via left--invariant vector fields, then the infinitesimal CR automorphism corresponds at $y=\exp(Y),Y\in \fn$ to $(Ad(\exp(-Y))(A))_{\fn}$, where $(Z)_{\fn}$ is the component of $Z\in \fg$ in $(\fg_-\oplus \fk)\otimes \mathbb{C}$ projected along $\fg_{-1,-1}\oplus \fg_{0,-1}$ into $\fn$. In the logarithmic coordinates, we can apply the Baker--Campbell--Hausdorff formula on 
\begin{align*}
&\frac{d}{dt}|_{t=0}\exp(Y)\exp(t(Ad(\exp(-Y))(A))_{\fn})=\\
&\frac{d}{dt}|_{t=0}\exp(tAd(\exp(Y))(Ad(\exp(-Y))(A))_{\fn})\exp(Y)
\end{align*}
to obtain a formulas for the infinitesimal CR automorphisms $A_i\in \fg_i$:
\begin{align*}
&\frac{d}{dt}|_{t=0}g(tA_{i},Y), {\rm \ for\ }i<-1\\
&\frac{d}{dt}|_{t=0}g(tAd(\exp(Y))(Ad(\exp(-Y))(A_i))_{\fn},Y), {\rm \ for\ }i\geq -1\\
\end{align*}
where $g(Z,W):=Z+\sum_{n=1}^{\infty}\frac{(-1)^{n}}{n+1}\sum_{s_i>0}\frac{ad(W)^{s_1+\dots+s_{n}}(Z)}{\prod_{i=1}^{n} s_i!}$ is the part of $Z+W+f(Z,W)$ linear in $Z$.
\end{proof}

\section{Equivalence problem for the regularly 2--nondegenerate CR geometries with simple models}\label{sec4}

In order to finish the proof of Theorem \ref{main}, we need to show that $dim(\fg)$ bounds the dimension of Lie algebra of infinitesimal CR automorphisms of all regularly 2--nondegenerate CR geometries with the same second--order Levi--Tanaka algebra $(\fg_-,I,\fk)$ as the model $G/H$. To do this, we solve the equivalence problem, i.e., we prove the Theorem \ref{abspar}.

One direction in the proof of Theorem \ref{abspar} is simple. We just need to pick normalization conditions that ensure that the projections of $\om^{-1}(\fg_{-1}\oplus \fk)$ at all points of $\ba$ determine a regularly 2--nondegenerate CR geometry on $M$ with the second--order Levi--Tanaka algebra $(\fg_-,I,\fk)$. Then we obtain a functor from the category of absolute parallelisms in Theorem \ref{abspar} to the category of 2--nondegenerate CR geometries with second--order Levi--Tanaka algebra $(\fg_-,I,\fk)$.

To prove the converse direction, we need to provide a construction of a unique $\om$ for each regularly 2--nondegenerate CR geometry $M$ with second--order Levi--Tanaka algebra $(\fg_-,I,\fk)$ and the lifts of CR morphisms to morphisms of the absolute parallelisms. In fact, we consider two constructions of $\om$ that use the same normalization conditions:

\begin{itemize}
\item In Sections \ref{sec42}, we present a technical construction of $\om$ that proves the Theorem \ref{abspar}. Part of this construction uses the Tanaka's prolongation theory from \cite{Ta70} and we avoid as many technicalities as possible because we also present a direct construction of $\om$.
\item In Sections \ref{sec43}, we present a direct construction of $\om$, the normalization conditions and construction of the lifts of CR morphisms. We remark that if we would like to use this construction to prove the Theorem \ref{abspar}, then we would need to show that the normalization conditions can be always satisfied and that they provide unique $\om$ and that the lifts of CR morphisms exist and preserve $\om$. This is hard to check directly and for this reason, we use the first construction to prove the Theorem \ref{abspar}.
\end{itemize}

The starting point for both of the constructions of $\om$ is the graded frame bundle $(\ba_0\to M,\theta),$ where $\ba_0\to M$ is the bundle of graded isomorphisms of the second--order Levi--Tanaka algebras at points of $M$ with the second--order Levi--Tanaka algebra $(\fg_-,I,\fk)$ and $\theta: T\ba\to \fg_-\oplus \fk$ is the natural soldering form provided by these isomorphisms.

The next step in both of the construction requires the following information from the theory of parabolic geometries.

\begin{lem*}\label{knf}
Suppose that $G,P,H,\fg_{a,b}$ are one of the cases from our classification in Tables \ref{realclasA}, \ref{realclasB}. Then

\begin{enumerate}
\item If $\fg_-$ is not a Heisenberg algebra, then $\fg$ is the Tanaka prolongation of $\fg_-\oplus \fg_0.$
\item If $\fg_-$ is a Heisenberg algebra, then $\fg$ is the Tanaka prolongation of $\fg_-\oplus \fg_0\oplus \fg_1,$ where $\fg_-\oplus \fg_0\oplus\fg_1$ is the maximal $G_0$--invariant subspace of the first prolongation of $\fg_-\oplus \fg_0$ consisting of maps preserving $I$ on $\fg_{-1}\oplus \fk$.
\item The intersection $H\cap G_0$ is reductive and coincides with the reduction of the group $Aut_0(\fg_-)$ to the subgroup $G_{0,I}$ of elements preserving $I$ on $\fg_{-1}\oplus \fk$ and $\fk$, i.e., it is the structure group of the graded frame bundle $\ba_0$.
\item If $\fg_-\otimes \mathbb{C}\oplus \fg_{0,-1}$ does not correspond to the case $\mathfrak{sp}(2n,\mathbb{C})\otimes \mathbb{C},\Sigma_1=\{1\},\Sigma_2=\{l\}$, then $\fg\otimes \mathbb{C}$ is the Tanaka prolongation of $\fg_-\otimes \mathbb{C}\oplus \fg_{0,-1}$.
\item If $\fg_-\otimes \mathbb{C}\oplus \fg_{0,-1}$ corresponds to the case $\mathfrak{sp}(2n,\mathbb{C})\otimes \mathbb{C},\Sigma_1=\{1\},\Sigma_2=\{l\}$, then $\fg\otimes \mathbb{C}$ is the Tanaka prolongation of $\fg_-\otimes \mathbb{C}\oplus \fg_{0,-1}\oplus \fg_{0,0}$.
\item The $0,1$--homogeneity parts of the cohomologies $H^1(\fg_-\otimes \mathbb{C}\oplus \fg_{0,-1},\fg\otimes \mathbb{C}),H^2(\fg_-\otimes \mathbb{C}\oplus \fg_{0,-1},\fg\otimes \mathbb{C})$ vanish.
\end{enumerate}
\end{lem*}
\begin{proof}
If we compare our classification with \cite[Proposition 4.3.1]{parabook}, then we get the first two claims. In the case that $\fg_-$ is the Heisenberg algebra, then the additional maps in the first prolongation of $\fg_-\oplus \fg_0$ are given by the cohomology $H^1(\fg_-,\fg)$ of homogeneity $1$. We can compute $H^1(\fg_-,\fg)$ (see \cite[Section 3.3]{parabook}) and see that it is an irreducible $G_0$--module with lowest weight in $\fg_{-1,0}^*\otimes \fg_{0,1}$, which is a map that does not preserve $I$ on $\fg_{-1}\oplus \fk$.

Another consequence of  \cite[Proposition 4.3.1]{parabook} is that only in the case $\mathfrak{sp}(2n,\mathbb{C})\otimes \mathbb{C},\Sigma_1=\{1\},\Sigma_2=\{l\}$ the Claim (3) does not follow directly. However, in this case the cohomology $H^1(\fg_-\otimes \mathbb{C}\oplus \fg_{0,-1},\fg\otimes \mathbb{C})$ of homogeneity $0$ is again a irreducible module with lowest weight in $\fg_{-1,0}^*\otimes \fg_{0,1}$, which is a map that does not preserve $I$ on $\fg_{-1}\oplus \fk$. This proves the Claims (4) and (5).

Moreover, this means that all first cohomologies $H^1(\fg_-\otimes \mathbb{C}\oplus \fg_{0,-1},\fg\otimes \mathbb{C})$ of homogeneity $0,1$ vanish. The second cohomologies $H^2(\fg_-\otimes \mathbb{C}\oplus \fg_{0,-1},\fg\otimes \mathbb{C})$ of homogeneity $0,1$ were explicitly computed in \cite{GZ16} and we see that they vanish in the cases from our classification in Table \ref{class}.
\end{proof}

\begin{rem*}\label{remsem}
The claims (1)--(4) of Lemma \ref{knf} can be also easily established without assumption $\fg$ is real simple, it is sufficient to assume that $\fg$ is semisimple. However, the cohomologies from claims (5)--(6) need to be computed, before we know, whether we can continue with our generalization of Tanaka's prolongation procedure. Since there are too many possible situations for semisimple $\fg$, we will do it in a separate article. Nevertheless, the analogies of claims (1) and (2) answer, when the dimension of $\fg$ is an algebraic bound for the maximally symmetric model. Indeed, there is always projection $G/H\to G/P$ and the CR automorphisms descend to automorphisms of the parabolic geometry on $G/P$ and claims (1) and (2) of Lemma \ref{knf} would imply the maximality.
\end{rem*}

\subsection{A construction using a distinguished partial connection on the Levi kernel}\label{sec42}

In this Section, we finished the proof of the Theorem \ref{abspar} in two steps: Firstly, we construct an infinitesimal $\fg_0$--structure on a regular filtered manifold with symbol $\fg_-$ on the graded frame bundle $(\ba_0\to M,\theta)$. Let us recall that an infinitesimal $\fg_0$--structure on a regular filtered manifold with symbol $\fg_-$ on the graded frame bundle $(\ba_0\to M,\theta)$ is a restriction $\theta_r$ of $\theta$ to preimige of $TM/\mathcal{K}$ together with the $\fg_0$--valued form $\theta_0$ on the preimage of $\mathcal{K}$ such that any extension of $\theta_r+\theta_0$ to a Cartan connection of type $(\exp(\fg_-)\rtimes G_0,G_{0,I})$ by forms of positive homogeneity w.r.t. to the grading $\fg_-\oplus \fg_0$ has curvature of positive homogeneity (w.r.t. to the grading $\fg_-\oplus \fg_0$). Such extensions exist due to Claim (3) of Lemma \ref{knf}.

In terms of the Tanaka's prolongation theory from \cite{Ta70} this is an infinitesimal version of pseudo--$G_0$--structure of type $\fg_-$ from \cite[Definitions 7.2 and 7.5]{Ta70}, where instead of the action of $G_0$ on $\ba_0$ there is only an infinitesimal $\fg_0$--action on $T\ba_0$. Nevertheless, this is enough for the construction of Tanaka's prolongation from \cite[Section 9]{Ta70}, because the fact that the action of $G_0$ is globally defined is not required in the construction. Therefore, the further steps of the construction of $\om$ are done as in \cite[Section 9]{Ta70} and we just need to check that we obtain a $H$--fiber bundle $\ba\to M$ with a $\fg$--valued absolute parallelism $\om$.

The $\fg_0$--valued form $\theta_0$ on the preimage of $\mathcal{K}$ is equivalent to a partial connection on the Levi kernel (preserving $I$) and corresponds to a choice of $G_{0,I}$--invariant complement of the vertical bundle in the preimage of $\mathcal{K}$ in $\ba_0$. The existence of such partial connections is given again by Claim (3) of Lemma \ref{knf}. We show that we can choose a unique partial connection providing an infinitesimal $\fg_0$--structure.

\begin{prop*}\label{parccon}
There is a unique partial connection on the Levi kernel such that the form $\theta_0: T\ba_0\to \fg_0$ defines an  infinitesimal $\fg_0$--structure on a regular filtered manifold with symbol $\fg_-$ on the graded frame bundle $(\ba_0\to M,\theta)$. This connection and $\theta_0$ are preserved by all CR morphisms.
\end{prop*}
\begin{proof}
We extend the form $\theta_0$ to a Cartan connection $\om_0$ on $\ba_0$ that as a graded isomorphism $T_u\ba_0\to \fg_-\oplus \fg_0$ coincides with $\theta_0(u)$ and $\theta_r(u)$ for all $u\in \ba_0$ and we want to show that there is $\theta_0$ such that the curvature (function) $\kappa: \ba_0\to \wedge^2(\fg_-\oplus \fg_0)^*\otimes \fg_-\oplus \fg_0$ of $\om_0$ has only components of positive homogeneity w.r.t. to the grading of $\fg_-\oplus \fg_0$.

We complexify $\om_0$ and decompose everything according the bigrading $\fg_{a,b}$. We can conclude from the construction of the second--order Levi--Tanaka algebra that the components of nonpositive homogeneity w.r.t. to the grading of $\fg_-\oplus \fg_0$ are determined (via Jacobi identity) by components $\ba_0\to\fg_{-1,-1}^*\wedge \fg_{0,-1}^*\otimes \fg_{-1,-1}$ and $\ba_0\to\fg_{-1,0}^*\wedge \fg_{0,1}^*\otimes \fg_{-1,0}$ of the curvature of the complexification of $\om_0$. Therefore, we need to show that there is unique choice of a partial connection on the Levi kernel for which one of these components vanish (because they are conjugated). 

We prove that $\kappa_{0,1}: \ba_0\to \sum_{b\leq 0}\fg_{a,b}^*\wedge \fg_{0,-1}^*\otimes \fg_{a,b}$ is element of the cohomology $H^2(\fg_-\otimes \mathbb{C}\oplus \fg_{0,-1},\fg\otimes \mathbb{C}).$ This can be done by following the proof of \cite[Theorem 3.1.12]{parabook} that is using the Bianchy identity
$$\sum_{cycl}[s_1,\kappa(s_2,s_3)]-\kappa([s_1,s_2],s_3)+\kappa(\kappa(s_1,s_2),s_3)-(D_{s_1}\kappa)(s_2,s_3)=0$$
from \cite[Proposition 1.5.9]{parabook}, where $s_1,s_2,s_3:\ba\to (\fg_-\oplus \fg_0)\otimes \mathbb{C}$ are $G_{0,I}$--equivariant maps and $D$ is the fundamental derivative.

In particular, if we follow the proof of \cite[Theorem 3.1.12]{parabook}, then we see that we need to show that the homogeneity $0,1$--parts of the Bianchy identity coincide with the algebraic differential in $(\fg_-\oplus \fg_0)\otimes \mathbb{C}$. The only difference in the proof of this fact is that $\kappa(s_1,s_2)$ can have entries in $\fg_{0,1}$ and thus we need different argument to show that $\kappa(\kappa(s_1,s_2),s_3)$ does not have homogeneity $0,1$--part. In this case, $\kappa_{0,-1}$ is the lowest homogeneity component of the first $\kappa$ in the expression thus $\kappa(\kappa(s_1,s_2),s_3)$ has nontrivial homogeneity $0,1$--part only for $s_1,s_2$ with values in $\fg_{0,-1}$. This is a contradiction with assumption (1) of Definition \ref{def2dg}. Thus $\kappa_{0,1}$ is closed and thus exact by Claim (6) of Lemma \ref{knf}. Therefore, the claimed partial connection exists.

The uniqueness  of the claimed partial connection follows from the vanishing of the cohomology $H^1(\fg_-\otimes \mathbb{C}\oplus \fg_{0,-1},\fg\otimes \mathbb{C})$ of homogeneity $0,1$ from Claim (6) of Lemma \ref{knf}. The uniqueness implies that this connection and $\theta_0$ are preserved by all (lifts of) CR morphisms on $\ba_0$.
\end{proof}

We can proceed with the construction of the absolute parallelism using Tanaka's prolongation theory from \cite{Ta70}. 

If $\fg_-$ is not a Heisenberg algebra, then the Claim (1)  of Lemma \ref{knf} ensures that the absolute parallelism constructed according to \cite{Ta70} proves the Theorem \ref{abspar}.

If $\fg_-$ is a Heisenberg algebra, then the Claim (2)  of Lemma \ref{knf} shows that we need an additional reduction to $\fg_1$. After the first prolongation according  to \cite{Ta70}, we obtain a bundle $\ba_1'$ over $\ba_0$ and the points of the fibers together with the components of $\om$ constructed in this step provide isomorphisms of $T^{-1}M$ at the underlying point with $\fg_-\oplus \fk$. We can make a reduction of $\ba_1'$ to points providing complex linear isomorphisms and obtain a reduction $\tilde \ba_1\subset \ba_1'$ to certain structure group $\exp(\tilde \fg_1)$. This implies that the complex antilinear part of the torsion $\fg_{-1}^*\wedge \fk^*\otimes \fk$ in the previous prolongation step depends algebraically on the point in the fiber of the bundle $\tilde \ba_1\to \ba_0$ and corresponds via the structure equation to part of $\tilde \fg_1$ such that $[\tilde \fg_1,\fk]$  has nontrivial component outside $\tilde \fg_1$. If we normalize this torsion to $0$, then we are left (due reductivity) with the maximal $G_0$--invariant subspace of $\tilde \fg_1$, i.e., $\fg_1$. Therefore, the Claim (2) of Lemma \ref{knf} ensures that the absolute parallelism constructed according to \cite{Ta70} (continuing after the reduction to $\fg_1$) proves the Theorem \ref{abspar}.

\subsection{Explicit construction of the absolute parallelism and the normalization conditions}\label{sec43}

Since $G_{0,I}$ is reductive and the normalization conditions can be chosen $G_{0,I}$--invariant, we will work in this section with $G_{0,I}$--equivariant functions on the graded frame bundle $\ba_0$ with values in representations of $G_{0,I}$.

According to \cite[Proposition 5.1.1]{parabook}, there is a global smooth  $G_{0,I}$--equivariant section $\si: \ba_0\to \ba$ and the space of these section is an affine space modeled on the space of sections $(TM/\mathcal{K})^*=_{\si}\ba_0\times_{G_{0,I}}\fp_+$, where $=_{\si}$ means that the identification is as in the case of parabolic geometries dependent on $\si$. Moreover, $\si^*\om$ decomposes into $G_{0,I}$--invariant one forms on $\ba_0$ with values in $G_{0,I}$--submodules of $\fg$. If $\om$ is Cartan connection and $\hat \si=\si\circ r^{\phi}$ is a section $\ba_0\to \ba$ given by a smooth $G_{0,I}$--invariant function $\phi: \ba_0\to \exp(\fp_+)$, then
$$\hat \si^*\om=Ad(\phi^{-1})(\si^*\om)+\delta \phi,$$
where $\delta \phi$ is the left logarithmic derivative of  $\phi: \ba_0\to \exp(\fp_+)$. In general, there is the difference $$Ad(\phi^{-1})\Theta(\phi):=\hat \si^*\om-(Ad(\phi^{-1})(\si^*\om)+\delta \phi),$$ which is a family of $G_{0,I}$--invariant one forms on $\ba_0$ with values in $G_{0,I}$--submodules of $\fg$ parametrized by the function $\phi$.

For the pullback $\si^*(d\om+[\om,\om])$ of the structure equations of $\om$, we get that 
$$\hat \si^*(d\om+[\om,\om])-Ad(\phi^{-1})(\si^*(d\om+[\om,\om]))=Ad(\phi^{-1})(d\Theta+[\Theta,\Theta]+[\si^*\om,\Theta]+[\Theta,\si^*\om]).$$
Therefore, if we normalize parts of $\si^*(d\om+[\om,\om])$ that are not $Ad(\phi^{-1})$--invariant, we obtain nontrivial correction terms $\Theta(\phi)$.

The component of complexification of $\si^*\om$ in $\fg_{a,b}$ can be written as $\om^{i,j}_{a,b}\theta_{i,j}$, where $\theta_{i,j}$ is complexification of $\theta$ and $\theta_{0,0}$ is determined by the Maurer--Cartan form of $G_{0,I}$ and by killing the homogeneity $0,1$ part of the pullback $\si^*(d\om+[\om,\om])$ according to Proposition \ref{parccon}. We need to choose normalization conditions that assign unique value to the $G_{0,I}$--equivariant maps $\om^{i,j}_{a,b}: \fg_{i,j}\to \fg_{a,b}$. We go through the components according to the homogeneity $k,l=-i+a,-j+b$ by increasing value of $k+l$:

\begin{enumerate}
\item[$k<0$] Since we are on filtered manifold, there are no such $\om^{i,j}_{a,b}$.
\item[$k=l=0$] Such $\om^{i,j}_{i,j}$ are the identity.
\item[$k=0,l\neq 0$]  Such $\om^{i,j}_{a,b}$ vanish by definition $\theta$ and construction of $\theta_{0,0}$.
\item[$k>0,\om_{k,l+1}^{0,1}$] These are conjugated to $\om_{k,l'}^{0,-1}$ for some $l'+1>0$. 
\item[$\om_{0,-1}^{-1,0},\om_{0,1}^{-1,-1}$]  Vanish, because are not compatible with $I$ on $\fg_{-1}\oplus \fk$.
\item[remaining]   We can normalize the remaining components according to Claims (4) and (5) of Lemma \ref{knf} as in the Tanaka's prolongation theory \cite{Ta70}. In particular, we can normalize them by a choice of the section $\si$ and by killing parts of pullbacks $\si^*(d\om+[\om,\om])$ that have entries in $\fg_{-1,0}$ or $\fg_{0,-1}$.
\end{enumerate}

Let us emphasize that the Tanaka's prolongation theory \cite{Ta70} does not provide the normalization conditions uniquely. It just ensures their existence. Moreover, it is not clear, whether $Ad(\phi^{-1})$--invariant normalization conditions exist, in general. We illustrate this on example in the following section.

The component $Re((\si^*\om)_{0,0})$ of $\si^*\om$ in $\fg_{0,I}$ is a principal connection on the graded frame bundle $\ba_0$ and for each CR morphism $\tau: \ba\to \ba'$ such that $\tau^*\om'=\om$, there is a section $\hat \si:\ba_0'\to\ba'$ such that $\tau_0^*Re((\hat \si^*\om')_{0,0})=Re((\si^*\om)_{0,0})$, where $\tau_0$ is the corresponding lift of the CR morphism to the graded frame bundle. Since the class of connections $Re((\si^*\om)_{0,0})$ is as in the case of parabolic geometries in bijective correspondence with the class of section $\si: \ba_0\to \ba$, this provides the lift of $\tau_0$ to $\tau$.

\subsection{Example --  Absolute parallelism for uniformly 2--nondegenerate hypersurfaces in $\mathbb{C}^3$}\label{exam}

Let us carry out our prolongation procedure for the case of uniformly 2--nondegenerate hypersurfaces $M$ in $\mathbb{C}^3$. This will demonstrate our results and allow them to be compared with the other prolongation procedures known in this case, see \cite{Eb06,IZ13,MS14,MS15,Poc13}. 

In this case, $\fg=\frak{sp}(4,\mathbb{R})$, $\Sigma_1=\{1\}$, i.e., $\fg_-$ is the Heisenberg algebra, $\Sigma_2=\{2\}$, i.e., $(H\cap G_0)/G_{0,I}=Gl(2,\mathbb{R})/CO(2).$ We consider $\frak{sp}(4,\mathbb{R})$ as the following real form of $\frak{sp}(4,\mathbb{C})$:
$$
\left[ \begin{array}{cccc}
e&\bar m& m& g\\
\bar l&f&\bar k& m\\
l& k& -f & -\bar m\\
j& l& -\bar l& -e
\end{array}\right]
\subset 
\left[ \begin{array}{cccc}
\fg'_{0,0}&\fg_{1,0}& \fg_{1,1}& \fg_{2,1}\\
\fg_{-1,0}&\fg_{0,0}&\fg_{0,1}& *\\
\fg_{-1,-1}& \fg_{0,-1}&* & *\\
\fg_{-2,-1}& *& *& *
\end{array}\right],
$$
where $e$ is real, $f,g,j$ are purely imaginary, $k,l,m$ are complex, $\fg_{i,j}$ indicate the bigrading of  $\frak{sp}(4,\mathbb{C})$ and $*$ means that entry is linearly dependent on the other entries (as in the first matrix). Let us emphasize that we use $'$ to distinguish between the two linearly independent parts of $\fg_{0,0}$.

Our choice of normalization conditions implies, that without loss of generality, we can choose a local section $s: M\to \ba_0$, i.e., local complex frame of $T^{-1}M/\mathcal{K}$. Locally, we work with the complexification of the pullback $s^*\si^*\om$, which is a $\frak{sp}(4,\mathbb{C})$--valued one form on $M$. These pullbacks extend $G_{0,I}$--equivariantly on local trivializations of $\ba_0\to M$ and can be glued to obtain $\si^*\om$ on $\ba_0$.

We will start with the following complex coframe $(j,l,\bar l,k,\bar k)$ that up to a multiple coincides with the coframe from \cite{Poc13} that defines a local section $s: M\to \ba_0$:
\begin{align*}
dj &=P j\wedge l+\bar P j\wedge \bar l-K_l j\wedge k-\bar K_{\bar l}j\wedge \bar k+2l\wedge \bar l,\\
 dl &= -K_j j\wedge k-K_l l\wedge k-K_{\bar l} \bar l\wedge k,\\
  d\bar l&= -\bar K_j j\wedge \bar k-\bar K_{\bar l} \bar l\wedge \bar k-\bar K_{l} l\wedge\bar k,\\
  dk &= 0, d\bar k = 0,
  \end{align*}
 where $K,P$ are two complex valued functions on $M$ (the formula for $K,P$ in terms of the defining equation for $M$ can be found in \cite{Poc13}) and $K_{i_1;i_2;\dots ;i_s}$ is a $s$--tuple of Lie derivatives of $K$ in directions $i_1,\dots , i_s$ w.r.t. to a nonholonomic frame dual to $(j,l,\bar l,k,\bar k)$. We remark that $K_{i_1;i_2;\dots ;i_s}$ depends on the ordering of the Lie derivatives and that in all our formulas bellow, the derivatives were ordered using the Jacobi identity.
 
If we decompose the complexification of the pullback $s^*\theta$ of the soldering form on $\ba_0$ according to the grading, then we obtain that
\begin{align*}
s^*\theta_{-2,-1}=j,\\
s^*\theta_{-1,-1}=l,\\
s^*\theta_{-1,0}=\bar l,\\
s^*\theta_{0,-1}=-K_{\bar l}k,\\
s^*\theta_{-1,0}=-\bar K_{l}\bar k.
\end{align*}
Therefore, the pullback $s^*\si^*\om$ has the following form:

\begin{gather*}
j\left[ \begin{array}{cccc}
0&\om^{-2,-1}_{1,0}& \om^{-2,-1}_{1,1}& \om^{-2,-1}_{2,1}\\
0&\om^{-2,-1}_{0,0}&\om^{-2,-1}_{0,1}& *\\
0& \om^{-2,-1}_{0,-1}&* & *\\
1& 0& 0& 0
\end{array}\right]\\
+l\left[ \begin{array}{cccc}
\om'{}^{-1,-1}_{0,0}&\om^{-1,-1}_{1,0}& \om^{-1,-1}_{1,1}& \om^{-1,-1}_{2,1}\\
0&\om^{-1,-1}_{0,0}&0& *\\
1& \om^{-1,-1}_{0,-1}&* & *\\
0& 0& 0& *
\end{array}\right]+\bar l\left[ \begin{array}{cccc}
\om'{}^{-1,0}_{0,0}&\om^{-1,0}_{1,0}& \om^{-1,0}_{1,1}& \om^{-1,0}_{2,1}\\
1&\om^{-1,0}_{0,0}&\om^{-1,0}_{0,1}& *\\
0& 0&* & *\\
0& 0& 0& *
\end{array}\right]\\
-K_{\bar l}k\left[ \begin{array}{cccc}
-\frac{K_l}{2K_{\bar l}}&\om^{0,-1}_{1,0}& \om^{0,-1}_{1,1}& \om^{0,-1}_{2,1}\\
0&-\frac{K_l}{2K_{\bar l}}&0& *\\
0&1&* & *\\
0& 0& 0& *
\end{array}\right]-\bar K_{l}\bar k\left[ \begin{array}{cccc}
-\frac{\bar K_{\bar l}}{2\bar K_{l}}&\om^{0,1}_{1,0}& \om^{0,1}_{1,1}& \om^{0,1}_{2,1}\\
0&\frac{\bar K_{\bar l}}{2\bar K_{l}}&1& *\\
0& 0&* & *\\
0& 0& 0& *
\end{array}\right]\\
\end{gather*}

Note, that the choice $\om^{-2,-1}_{-1,0}=0,\om^{-2,-1}_{-1,-1}=0,\om'{}_{-1,-1}^{-2,-1}=0$ removes the freedom in choice of the section $\phi$, $\om^{-1,-1}_{0,1}=0$, $\om^{-1,0}_{0,-1}=0$ vanish by definition and the choice of $\om'{}^{0,-1}_{0,0},\om^{0,-1}_{0,0},\om'{}^{0,1}_{0,0},\om^{0,1}_{0,0}$ determines the distinguished connection on the Levi kernel from Proposition \ref{parccon}. Moreover, there are nontrivial relations between the remaining entries of the matrices provided by the fact that we obtained this by complexification that can be deduced by the inclusion $\frak{sp}(4,\mathbb{R})\subset \frak{sp}(4,\mathbb{C})$.

Now, we can start the normalization procedure and determine the pullback $s^*\si^*\om$ and the correction term $\Theta(\phi)$. So we look on components of $R:=s^*\si^*(d\om+[\om,\om])$ and write $R^{a,b}_{c,d;e,f}$ for the component of $R$ in $\fg_{a,b}$ evaluated on the vector fields providing the dual frame to the coframe $$(j,l,\bar l,-K_{\bar l}k+l\om^{-1,-1}_{0,-1}+j\om^{-2,-1}_{0,-1},-\bar K_{l}\bar k+\bar l \om^{-1,0}_{0,1}+j\om^{-2,-1}_{0,1} ),$$ where $c,d$ and $e,f$ indicate the component of $s^*\si^*\om$ providing the nonvanishing (on the vector field) part of the coframe. We want to normalize suitable components $R^{a,b}_{c,d;-1,0}$ and $R^{a,b}_{c,d;0,-1}$. Let proceed according to homogeneity of $R^{a,b}_{c,d;e,f}$.

\begin{enumerate}
\item[$1,0$] We can normalize $R^{-2,-1}_{-2,-1;-1,0}=R^{-1,-1}_{-1,-1;-1,0}=R^{0,-1}_{-1,-1;0,-1}=0$ and obtain:
\begin{align*}
\om^{-1,-1}_{0,-1}&=\bar \om^{-1,0}_{0,1}=\frac{\bar P K_{\bar l}-K_{\bar l;\bar l}}{3K_{\bar l}}\\
\om'{}^{-1,0}_{0,0}&=\bar \om'^{-1,-1}_{0,0}=-\frac{\bar P}{2}\\
\om^{-1,0}_{0,0}&=-\bar \om^{-1,-1}_{0,0}=\frac{\bar P K_{\bar l}+2K_{\bar l;\bar l}}{6K_{\bar l}}\\
\end{align*}
\item[$1,1$] We can normalize $R^{-1,0}_{-2,-1;0,-1}=0$ and obtain:
\begin{align*}
\om^{0,-1}_{1,0}&=-\bar \om^{0,1}_{1,1}=-\frac{K_{j}}{K_{\bar l}}\\
\end{align*}
\item[$1,2$] We can normalize $R^{-1,-1}_{-2,-1;0,-1}=0$ and obtain:
\begin{align*}
\om^{0,-1}_{1,1}&=\bar \om^{0,1}_{1,0}=0\\
\end{align*}
\item[$W$] There are several remaining components of $R$ in the homogeneities $1,1$ and $1,0$ of the form:
\begin{align*}
-3R^{0,-1}_{-1,-1;0,-1}&=3R^{0,0}_{-1,0;0,-1}=W=\frac{4K_{j}}{K_{\bar l}}-\frac{4K_{l;\bar l}}{K_{\bar l}}-\frac{2\bar K_{l;l}}{\bar K_{l}}+\frac{K_{\bar l;\bar l; k}}{K_{\bar l}^2}-\frac{K_{\bar l;\bar l}K_{\bar l;k}}{K_{\bar l}^3}\\
-3R^{0,0}_{-1,-1;0,1}&=-3R^{0,1}_{-1,0;0,1}=\bar W,\\
\end{align*}
where $W$ is the first fundamental invariant. Note that the ``$0,0$'' position is purely imaginary in the real form and thus the sign change.
\item[$2,0$] We can normalize $R^{-1,-1}_{-2,-1;-1,0}=R^{0,-1}_{-1,-1;-1,0}=0$ and obtain:
\begin{align*}
\om^{-2,-1}_{0,-1}&=-\bar \om^{-2,-1}_{0,1}=\frac{\bar P^2 K_{\bar l}^2+\bar PK_{\bar l}K_{\bar l;\bar l}-3K_{\bar l}^2\bar P_{\bar l}+3K_{\bar l}K_{\bar l;\bar l;\bar l}-5K_{\bar l;\bar l}^2}{36K_{\bar l}^2}\\
\om^{-1,0}_{1,0}&=\bar \om^{-1,-1}_{1,1}=-\frac{\bar P^2 K_{\bar l}^2+\bar PK_{\bar l}K_{\bar l;\bar l}-3K_{\bar l}^2\bar P_{\bar l}+3K_{\bar l}K_{\bar l;\bar l;\bar l}-5K_{\bar l;\bar l}^2}{36K_{\bar l}^2}\\
\end{align*}
\item[$2,1$] We can normalize $R^{-1,0}_{-2,-1;-1,0}=R^{0,0,'}_{-1,-1;-1,0}=2R^{1,0}_{-1,0;0,-1}+2R^{1,1}_{-1,-1;0,1}+R^{0,0}_{-1,-1;-1,0}=0$ and obtain:
\begin{align*}
\om^{-2,-1}_{0,0}&=\om^{-1,-1}_{1,0}=\om^{-1,0}_{1,1}=
 -\frac{P\bar P}{36}-\frac{P\bar W+\bar PW}{72}-\frac{P\bar K_{j}}{6\bar K_{l}}+\frac{(2P+W)K_{\bar l;\bar l}}{36K_{\bar l}}+\frac{(2\bar P+\bar W)\bar K_{l;l}}{36\bar K_{l}}\\
& +\frac{W_{\bar l}+\bar W_{l}}{24}+\frac{P_{\bar l}}{12}+\frac{\bar K_{l;l}\bar K_{j}}{6\bar K_{l}^2}
+\frac{K_{l;\bar l;\bar l}}{12K_{\bar l}}-\frac{K_{l;\bar l}K_{\bar l;\bar l}}{12K_{\bar l}^2}-\frac{\bar K_{j;l}}{3\bar K_{l}}+\frac{\bar K_{l;l;\bar l}}{12\bar K_{l}}-\frac{\bar K_{l;l}K_{\bar l;\bar l}}{36\bar K_{l}K_{\bar l}}-\frac{\bar K_{l;l}\bar K_{l;\bar l}}{12\bar K_{l}^2}
\end{align*}
\item[$2,2$] We can normalize $R^{0,0,'}_{-2,-1;0,-1}=0$ and obtain:
\begin{align*}
 \om^{0,-1}_{2,1}=\bar  \om^{0,1}_{2,1}=-\frac{(PK_{j}+K_{j;l})}{2K_{\bar l}}
\end{align*}
\item[$2,.$] We can not normalize any other component in homogenity $2,.$ and the remaining components of $R$ does not provide any new fundamental invariants. In particular, these components vanish when $W=0$, so we do not write them down explicitly.
\item[$J$] In homogeneity $3,0$, we find the second  fundamental invariant:
\begin{align*}
12R^{0,-1}_{-2,-1;-1,0}&=\bar J=\frac{4P^3}{9}+\frac{2P^2\bar K_{l;l}}{3\bar K_{l}}-2PP_l+\frac{P\bar K_{l;l;l}}{\bar K_{l}}-\frac{5P\bar K_{l;l}^2}{3\bar K_{l}^2}+P_{l;l}-\frac{P_l\bar K_{l;l}}{\bar K_{l}}\\
&-\frac{\bar K_{l;l;l;l}}{\bar K_{l}}+\frac{5\bar K_{l;l}\bar K_{l;l;l}}{\bar K_{l}^2}-\frac{40\bar K_{l;l}^3}{9\bar K_{l}^3}.
\end{align*}
Since we have all fundamental invariants, we continue with this example without explicit formulas that are too long to be presented here.
\item[$3,1$] There are two possible choices of normalization. 

Either we can normalize $R^{0,0,'}_{-2,-1;-1,0}=R^{0,0}_{-1,-1;-1,0}=0$, which depend on $\phi$, and obtain:
\begin{align*}
 \om^{-2,-1}_{1,1},\om^{-2,-1}_{1,0},\om^{-1,-1}_{2,1},\om^{-1,0}_{2,1},
\end{align*}
and the corresponding correction terms $\Theta(\phi)$. 

Or we can normalize $$R^{0,0,'}_{-2,-1;-1,0}=-R^{0,-1}_{-2,-1;-1,-1}+R^{0,0}_{-2,-1;-1,0}+R^{1,1}_{-2,-1;0,1}=0,$$ which does not depend on $\phi.$ However, $-R^{0,-1}_{-2,-1;-1,-1}+R^{0,0}_{-2,-1;-1,0}+R^{1,1}_{-2,-1;0,1}=0$ is a differential equation of the form:
$$(\om^{-1,-1}_{2,1})_{\bar k}=2K_l\bar \om^{-1,-1}_{2,1}+K_{\bar l}\om^{-1,-1}_{2,1}+F(K,P),$$
for certain nonlinear differential operator $F$ acting on the functions $K$, $P$. 

If $W=0$, then there is normalization condition $\partial^*R=0$, where $\partial^*$ is the usual codifferential in the theory of parabolic geometries, see \cite[Section 3.3]{parabook}.
\item[$4,2$] In homogeneity $4,.$, we no longer (unless $W=0$) have an option to choose linear combination of components of $R$ that would not depend on $\phi$. We can normalize $R^{2,1}_{-1,-1;-1,0}=0$ and obtain the remaining:
\begin{align*}
 \om^{-2,-1}_{2,1}
\end{align*}
and the corresponding correction term $\Theta(\phi)$. 
\end{enumerate}

This does not rule out the existence of absolute parallelism that would not depend on $\phi$ (a Cartan connection). It only shows that the normalization condition has to be given by a differential operator $D$ acting on $R$ such that $D(R)$ does not depend on $\phi$ and such that $D(R)=0$ uniquely determines the absolute parallelism/Cartan connection.

\section{Defining equations for the hypersurface models}\label{sec5}
\addtocontents{toc}{\setcounter{tocdepth}{1}}

In this section, we use the Proposition \ref{embed} to compute the defining equations for all models from our classification in Table \ref{realclasA} that are of codimension $1$.

In general, the computation proceeds as follows:

We represent elements of $\fg\otimes \mathbb{C}$ as block matrices decomposed according to the bigrading $\fg_{a,b}$. We specify the elements $X\in \fg_-$, where we decompose the blocks into the real part and imaginary parts $X^j+iY^j$ of $\fg_{-1}$ and represent $\fg_{-2}$ by $c\in \mathbb{R}$. This realizes the real form $\fg$ in $\fg\otimes \mathbb{C}$ in agreement with the Proposition \ref{embed}. The different real forms are distinguished by the diagonal matrices $I_{a,b}$ with $a$ and $b$ being the numbers $-1$ and $1$ on the diagonal. Further, we specify the elements $\frac12(Y-iI(Y))\in \fg_{0,-1}$ for $Y\in \fk$ and use the notation $\Xi$ for the nontrivial blocks.

Since we are working with nilpotent matrices, we can directly compute the embedding $\phi(X+Y)$ using the Taylor expansions for $\exp$ and $\exp^{-1}$ and matrix multiplication. We use the notation $Z^j$ for the blocks in $\fg_{-1,-1}$ and represent $\fg_{-2,-1}$ by $w\in \mathbb{C}$ .

Finally, we eliminate the blocks $X^j,Y^j$ and $c$ from $\phi(X+Y)$ and we present a single defining equation given by combinations of multiplication of the blocks $\Xi,Z^j$ and their conjugates $\bar \Xi,\bar Z^j$. We remark that we use the notation $\bar \Xi^T$ for the conjugate transpose of $\Xi$.

We name the subsection according to the entries of our classification in Table \ref{realclasA} and specify the signature of the Levi form to distinguish the cases with different $\fg_{0,I}'$. The first case is presented in greater detail.

\subsection{$\fg=\frak{sl}(2n+2,\mathbb{R}),\Sigma_1=\{\alpha_1,\alpha_{2n+1}\}, \Sigma_2=\{\alpha_{n+1}\}, sgn=(n,n,n^2)$}

The first entry in the Table \ref{realclasA} has $\fg=\frak{sl}(2n+2,\mathbb{R})$ and provides contact grading for $\Sigma_1=\{\alpha_1,\alpha_{2n+1}\},\Sigma_2=\{\alpha_{n+1}\}$. This implies $\fg_{-1,-1}=\mathbb{C}^n\oplus (\mathbb{C}^n)^*,\fg_{0,-1}=\mathbb{C}^n\otimes (\mathbb{C}^n)^*, sgn=(n,n,n^2)$ and
$$X= \left[ \begin {array}{cccc} 0&0&0&0\\ X^1-iY^1&0&0&0\\ X^1+iY^1&0&0&0
\\ c&X^2+iY^2&X^2-iY^2&0
\end {array} \right] , Y= \left[ \begin {array}{cccc} 0&0&0&0\\ 0&0&0&0
\\ 0&\Xi&0&0\\ 0&0&0&0\end {array}
 \right], 
$$
where $X^1,Y^1\in \mathbb{R}^n, X^2,Y^2\in (\mathbb{R}^n)^*$, $\Xi\in \mathbb{C}^n\otimes (\mathbb{C}^n)^*$.
Now, we apply the formula from Proposition \ref{embed} on the matrices $X,Y$ and obtain the formula for the embedding
$$\phi(X+Y)= \left[ \begin {array}{cccc} 0&0&0&0\\ 0&0&0&0\\ Z^1&\Xi&0&0\\ w&Z^2&0&0\end {array} \right],  \begin{array}{c}
Z^1=(\id+\Xi) X^1+ i(\id-\Xi) Y^1, \\
Z^2=X^2(\id+\Xi) + iY^2(\id-\Xi) ,\\
w=c+( -X^2\Xi
-iY^2(\id-\Xi) )X^1\\
+ (  iX^2(\id+\Xi) +Y^2\Xi)Y^1
\end {array},
$$
where $Z^1\in \mathbb{C}^n,Z^2\in (\mathbb{C}^n)^*$.
We can solve that 
\begin{align*}
X^1&=(2\id-\Xi \bar \Xi-\bar \Xi\Xi)^{-1}((\id+\Xi)\bar Z^1+(\id+\bar \Xi)Z^1),\\
X^2&=(\bar Z^2(\id-\Xi)+Z^2(\id-\bar \Xi))(2\id-\Xi \bar \Xi-\bar \Xi\Xi)^{-1},\\
Y^1&=i(2\id-\Xi \bar \Xi-\bar \Xi\Xi)^{-1}((\id-\Xi)\bar Z^1-(\id-\bar \Xi)Z^1),\\
Y^2&=i(\bar Z^2(\id+\Xi)-Z^2(\id+\bar \Xi))(2\id-\Xi \bar \Xi-\bar \Xi\Xi)^{-1}
\end{align*} and plug it into formula for $Im(w)$ to obtain the following defining equation
\begin{align*}
Im(w)&=-i(Z^2(2\id-\Xi \bar \Xi-\bar \Xi\Xi)^{-1}\bar Z^1-\bar Z^2(2\id-\Xi \bar \Xi-\bar \Xi\Xi)^{-1}Z^1)\\
&+Im(Z^2(2\id-\Xi \bar \Xi-\bar \Xi\Xi)^{-1}\bar \Xi Z^1).
\end{align*}
In the case $n=1$, the signature is $(1,1,1)$ and the defining equation simplifies to
\begin{align*}
Im(w)&=\frac{-i(\bar z_1 z_2-\bar z_2 z_1)+Im(2z_1z_2\bar \xi)}{2\id-2\xi \bar \xi},
\end{align*}
where $z_1,z_2,\xi\in\mathbb{C}$.

\subsection{$\fg=\frak{su}(p+1,q+1),\Sigma_1=\{\alpha_1,\alpha_{p+q+1}\},\Sigma_2=\{\alpha_{r+s+1}\}, sgn=(r+q-s,s+p-r,(r+s)(p+q-r-s))$}

$$X= \left[ \begin {array}{cccc} 0&0&0&0\\ I_{r,s}(X^2-iY^2)^T&0&0&0\\ 
X^1+iY^1&0&0&0\\ ic
&X^2+iY^2&(X^1-iY^1)^TI_{p-r,q-s}&0
\end {array} \right] 
, Y= \left[ \begin {array}{cccc} 0&0&0&0\\ 0&0&0&0
\\ 0&\Xi&0&0\\ 0&0&0&0\end {array} \right],
$$
where $X^1,Y^1\in \mathbb{R}^{p+q-r-s}, X^2,Y^2\in (\mathbb{R}^{r+s})^*$, $\Xi\in \mathbb{C}^{p+q-r-s}\otimes (\mathbb{C}^{r+s})^*$.
$$\phi(X+Y) \left[ \begin {array}{cccc} 0&0&0&0\\ 0&0&0&0\\ Z^1&\Xi&0&0\\ w&Z^2&0&0\end {array} \right],   \begin{array}{c}
Z^1=X^1+ iY^1-\Xi I_{r,s}(X^2- iY^2)^T, \\
Z^2=X^2+ iY^2+(X^1-iY^1)^TI_{p-r,q-s}\Xi, \\
w=ic-\frac12(X^2+ iY^2) I_{r,s} (X^2- iY^2)^T\\
+\frac12(X^1-iY^1)^T I_{p-r,q-s} (X^1+ iY^1)\\
-(X^1-iY^1)^TI_{p-r,q-s} \Xi I_{r,s} (X^2- iY^2)^T
\end{array},
$$
where $Z^1\in \mathbb{C}^{p+q-r-s},Z^2\in (\mathbb{C}^{r+s})^*$.

\begin{align*}
Re(w)&=\frac12(\bar Z^1)^TD_{p-r,q-s} Z^1-\frac12 Z^2D_{r,s}(\bar Z^2)^T\\
&+Re(\frac12 Z^2D_{r,s}\bar  \Xi^T I_{p-r,q-s}Z^1+\frac12Z^2 I_{r,s} \bar  \Xi^TD_{p-r,q-s} Z^1),
\end{align*}
where $D_{r,s}=(I_{r,s}+ \bar  \Xi^T I_{p-r,q-s} \Xi)^{-1}$ and $D_{p'',q''}=(I_{p-r,q-s}+ \Xi I_{r,s} \bar  \Xi^T)^{-1}$ are Hermitian matrices.

In the case $p=q=r=1,s=0$, the signature is $(2,0,1)$ and the defining equation simplifies to:
\begin{align*}
Re(w)&=-\frac{z_1 \bar z_1+z_2\bar z_2+Re(2 z_1z_2\bar  \xi)}{2-2\xi\bar \xi},
\end{align*}
where $z_1,z_2,\xi\in\mathbb{C}$.

Similarly, in the case $p=q=s=1,r=0$, the signature is $(0,2,1)$ and
\begin{align*}
Re(w)&=\frac{\bar z_1z_1+z_2\bar z_2-Re(2 z_1z_2\bar  \xi)}{2-2\xi\bar \xi}
\end{align*}
and in the case $p=2,q=0,r=1,s=0$, the signature is $(1,1,1)$ and
\begin{align*}
Re(w)&=\frac{\bar z_1z_1-z_2\bar z_2+Re(2 z_1z_2\bar  \xi)}{2+2\xi\bar \xi}
\end{align*}

\subsection{$\fg=\frak{so}(p+2,q+2),\Sigma_1=\{2\}, \Sigma_2=\{1\}, sgn=(q,p,1)$}

$$X=  \left[ \begin {array}{ccccc} 0&0&0&0&0\\ 0&0&0
&0&0\\ X^1+iY^1&X^1-iY^1&0&0&0
\\ ic&0&(X^1-iY^1)^TI_{p,q}&0&0\\ 
0&-ic&(X^1+iY^1)^TI_{p,q}&0&0\end {array} \right] 
, Y= \left[ \begin {array}{ccccc} 0&0&0&0&0\\ \xi&0&0&0&0\\ 0&0&0&0&0\\
0&0&0&0&0
\\ 0&0&0&-\xi&0\end {array}
 \right] ,
$$
where $X^1,Y^1\in \mathbb{R}^{p+q}$ and $\xi\in \mathbb{C}$.
$$\phi(X+Y)= \left[ \begin {array}{ccccc} 0&0&0&0&0\\ \xi&0&0
&0&0\\ Z^1&0&0&0&0
\\  w&0&0&0&0\\ 0&-w&(Z^1)^TI_{p,q}&-\xi&0\end {array} \right],
\begin{array}{c}
Z^1=X^1+iY^1+\xi(X^1-iY^1), \\
w=\frac12((X^1)^TI_{p,q}X^1+(Y^1)^TI_{p,q}Y^1)\\
+i c+\frac12 \xi ((X^1)^TI_{p,q}X^1-(Y^1)^TI_{p,q}Y^1)\\
\end{array},
$$
where $Z^1\in \mathbb{C}^{p+q}$.

\begin{align*}
Re(w)&=\frac{1}{2-2\xi\bar \xi}((\bar Z^1)^TI_{p,q}Z^1-Re(\bar \xi(\bar Z^1)^TI_{p,q}Z^1)).
\end{align*}

In the case, $p=0,q=1$,  the signature is $(1,0,1)$ and the defining equation simplifies (after reparametrization) to the well--know local rational model of a tube over a light cone:

\begin{align*}
Re(w)&=\frac{z \bar z+Re(z^2\bar \xi)}{1-\xi\bar \xi}.
\end{align*}

\subsection{$\fg=\frak{so}^*(2n+2),\Sigma_1=\{2\}, \Sigma_2=\{1\}, sgn=(n,n,1)$}

$$X= \left[ \begin {array}{cccccc} 0&0&0&0&0&0\\ 0&0&0&0
&0&0\\ X^1+iY^1&-X^2+iY^2&0&0&0&0
\\ X^2+iY^2&X^1-iY^1&0&0&0&0\\ c&0
&(X^2-iY^2)^T&(-X^1+iY^1)^T&0&0\\ 0&-c&-(X^1+iY^1)^T&-(X^2+iY^2)^T&0&0\end {array} \right] 
, $$
where $X^1,Y^1,X^2,Y^2\in \mathbb{R}^{n}$, and
$$Y= \left[ \begin {array}{cccccc} 0&0&0&0&0&0\\ \xi&0&0&0&0&0\\ 0&0&0&0&0&0
\\ 0&0&0&0&0&0\\ 0&0&0&0&0&0
\\ 0&0&0&0&-\xi&0\end {array}
 \right],
$$
where $\xi\in \mathbb{C}$.
$$\phi(X+Y)= \left[ \begin {array}{cccccc} 0&0&0&0&0&0\\ \xi&0&0&0
&0&0\\ Z^1&0&0&0&0&0
\\ Z^2&0&0&0&0&0
\\ w&0&0&0&0&0\\ 0&-w&-(Z^1)^T&-(Z^2)^T&-\xi&0\end {array} \right] 
,$$
where $Z^1,Z^2\in \mathbb{C}^n$, and
$$  \begin{array}{c}
Z^1=X^1+iY^1-\xi(X^2-iY^2), \\
Z^2=X^2+iY^2+\xi(X^1-iY^1), \\
w=c-i((Y^1)^T X^2 +(X^1)^T Y^2+ \xi ((X^1)^TY^1+(X^2)^TY^2)) \\
+\frac12 \Xi (-(Y^1)^T Y^1  - (X^2)^T X^2  -(Y^2)^T Y^2  -(X^1)^T X^1 )
\end{array}
$$

\begin{align*}
Im(w)&=\frac{1}{2+2\xi\bar \xi}(i(\bar Z^1)^TZ^2-i(\bar Z^2)^TZ^1-\frac12Im(\bar \xi((Z^1)^TZ^1+(Z^2)^TZ^2))).
\end{align*}
In the case, $n=1$,  the signature is $(1,1,1)$ and the defining equation simplifies to:
\begin{align*}
Im(w)&=\frac{i\bar z_1z_2-i\bar z_2z_1-\frac12Im(\bar \xi(z_1^2+z_2^2))}{2+2\xi\bar \xi}.
\end{align*}

\subsection{$\fg=\frak{so}(2p+2,2q+2),\Sigma_1=\{2\}, \Sigma_2=\{p+q+2\}, sgn=(2p,2q,\frac{p+q}{2}(p+q-1))$}

$$X=  \left[ \begin {array}{cccccc} 0&0&0&0&0&0\\ 0&0&0&0
&0&0\\ I_{p,q}(X^1-iY^1)&I_{p,q}(X^2-iY^2)&0&0&0&0
\\ X^1+iY^1&X^2+iY^2&0&0&0&0
\\ c&0&-(X^2+iY^2)^T&(-X^2+iY^2)^TI_{p,q}&0&0\\ 0&-c&-(X^1+iY^1)^T&(-X^1+iY^1)^TI_{p,q}&0&0\end {array} \right],$$ 
where $X^1,Y^1,X^2,Y^2\in \mathbb{R}^{p+q}$, and
$$Y= \left[ \begin {array}{cccccc} 0&0&0&0&0&0\\ 0&0&0&0&0&0\\ 0&0&0&0&0&0
\\ 0&0&\Xi&0&0&0\\ 0&0&0&0&0&0
\\ 0&0&0&0&0&0\end {array}
 \right] , \phi(X+Y)= \left[ \begin {array}{cccccc} 0&0&0&0&0&0\\ 0&0&0&0
&0&0\\ 0&0&0&0&0&0
\\ Z^1&Z^2&\Xi&0&0&0
\\ w&0&-(Z^2)^T&0&0&0\\ 0&-w&-(Z^1)^T&0&0\end {array} \right] 
,$$
where $\Xi\in \wedge^2 \mathbb{C}^{p+q}$ and $Z^1,Z^2\in \mathbb{C}^{p+q},$ and
$$  \begin{array}{c}
Z^1=X^1+iY^1-\Xi I_{p,q}(X^1-iY^1), \\
Z^2=X^2+iY^2-\Xi I_{p,q}(X^2-iY^2), \\
w=c+\frac12 (X^2+iY^2)^TI_{p,q} (X^1-iY^1)-\frac12 (X^2-iY^2)^T I_{p,q} (X^1+iY^1)\\
+(X^2-iY^2)^T I_{p,q}\Xi I_{p,q} (X^1-iY^1)
\end {array}
$$

\begin{align*}
Im(w)&=\frac12((\bar Z^2)^T D Z^1+(\bar Z^1)^T D Z^2+Im((Z^2)^T (D \bar \Xi I_{p,q} +I_{p,q} \Xi D) Z^1)),
\end{align*}
where $D=(I_{p,q}-\Xi I_{p,q}\bar \Xi)^{-1}$ is a Hermitian matrix.

\subsection{$\fg=\frak{so}^*(2n+2),\Sigma_1=\{2\}, \Sigma_2=\{n+1\},sgn=(2p,2(n-p),\frac{n}{2}(n-1))$}

$$X= \left[ \begin {array}{cccccc} 0&0&0&0&0&0\\ 0&0&0&0
&0&0\\ I_{p,n-p}(X^2-iY^2)&I_{p,n-p}(-X^1+iY^1)&0&0&0&0
\\ X^1+iY^1&X^2+iY^2&0&0&0&0\\ ic&0
&-(X^2+iY^2)^T&(X^1-iY^1)^TI_{p,n-p}&0&0\\ 0&-ic&-(X^1+iY^1)^T&(-X^2+iY^2)^TI_{p,n-p}&0&0\end {array} \right] 
, 
$$
where $X^1,Y^1,X^2,Y^2\in \mathbb{R}^{n}$, and
$$Y= \left[ \begin {array}{cccccc} 0&0&0&0&0&0\\ 0&0&0&0&0&0\\ 0&0&0&0&0&0
\\ 0&0&\Xi&0&0&0\\ 0&0&0&0&0&0
\\ 0&0&0&0&0&0\end {array}
 \right] ,\phi(X+Y)= \left[ \begin {array}{cccccc} 0&0&0&0&0&0\\ 0&0&0&0
&0&0\\ 0&0&0&0&0&0
\\ Z^1&Z^2&\Xi&0&0&0
\\ w&0&-(Z^2)^T&0&0&0\\ 0&-w&-(Z^1)^T&0&0&0\end {array} \right] 
, $$
where $\Xi\in \wedge^2 \mathbb{C}^{n}$ and $Z^1,Z^2\in \mathbb{C}^{n},$ and
$$ \begin{array}{c}
Z^1=X^1+iY^1-\Xi I_{p,n-p}(X^2-iY^2), \\
Z^2=X^2+iY^2+\Xi I_{p,n-p}(X^1-iY^1), \\
w=ic+\frac12((X^2+iY^2)^T I_{p,n-p}(X^2-iY^2)+(X^1-iY^1)^T I_{p,n-p}(X^1+iY^1))\\
-(X^1-iY^1)^T I_{p,n-p}\Xi I_{p,n-p}(X^2-iY^2)
\end {array}
$$

\begin{align*}
Re(w)&=\frac12((\bar Z^2)^T D Z^2+ (\bar Z^1)^T D Z^1-Re((Z^2)^T (D \bar \Xi I_{p,n-p} + I_{p,n-p}\bar \Xi D) Z^1)),
\end{align*}
where $D=(I_{p,n-p}+\bar \Xi I_{p,n-p}\Xi)^{-1}$ is a Hermitian matrix.

\subsection{A generic model $\fg=\frak{sp}(2n+2,\mathbb{R}),\Sigma_1=\{\alpha_1\},\Sigma_2=\{\alpha_{n+1}\}, sgn=(p,n-p,\frac{n(n+1)}{2})$}

Let us emphasize that $\fg_0$ is in this case the full algebra of infinitesimal contactomorphisms preserving the grading of $\fg_-$ and this is unique model for 2--nondegerate CR hypersurfaces with such signature of the Levi form.

$$X= \left[ \begin {array}{cccc} 0&0&0&0\\ I_{p,n-p}(X^1-
iY^1)&0&0&0\\ 
X^1+iY^1&0&0&0\\ ic
&(X^1+iY^1)^T&(X^1-iY^1)^TI_{p,n-p}&0
\end {array} \right] 
, Y= \left[ \begin {array}{cccc} 0&0&0&0\\ 0&0&0&0
\\ 0&\Xi&0&0\\ 0&0&0&0\end {array} \right],
$$
where $X^1,Y^1\in \mathbb{R}^n$ and $\Xi\in S^2\mathbb{C}^n$
$$\phi(X+Y)= \left[ \begin {array}{cccc} 0&0&0&0\\ 0&0&0&0\\ Z^1&\Xi&0&0\\ w&(Z^1)^T&0&0\end {array} \right],  \begin{array}{c}
Z^1=X^1+ iY^1-\Xi I_{p,n-p}(X^1- iY^1), \\
w=ic- (X^1- iY^1)^T I_{p,n-p} (X^1+ iY^1)\\
+ (X^1- iY^1)^T I_{p,n-p} \Xi I_{p,n-p}(X^1- iY^1)
\end{array},
$$
where $Z^1\in \mathbb{C}^n$.
\begin{align*}
Re(w)&=-(\bar Z^1)^TD_{p,n-p} Z^1-Re((Z^1)^T I_{p,n-p} \bar  \Xi D_{p,n-p} Z^1),
\end{align*}
where $D_{p,n-p}=(I_{p,n-p}- \Xi I_{p,n-p} \bar \Xi)^{-1}$ is a Hermitian matrix.

In the case $n=1,p=0$, the signature is $(0,1,1)$ and the defining equation simplifies to:
\begin{align*}
Re(w)&=\frac{z \bar z+Re(z^2\bar  \xi)}{1-\xi\bar \xi},
\end{align*}
which is again the well--know local rational model of a tube over a light cone.

\subsection{Remark on exceptional models}

In the Table \ref{realclasB}, there are several entries for the real forms of the exceptional Lie algebras $\frak{e}_6$ and $\frak{e}_7$ that that are of codimesion $1$. We do not compute the defining equations for these models due to computational complexity. We think that the computation of these remaining defining equations can be a good problem for a Bachelor/Master thesis.

\end{document}